\documentclass[11pt]{article}

\usepackage[small,bf]{caption}
\usepackage{amsmath, amsthm, amssymb}
\usepackage{algorithmic}
\usepackage{algorithm}

\usepackage{booktabs}

\usepackage{epsfig,fullpage,psfrag,url,verbatim}
\usepackage{graphics}
\usepackage{graphicx}

\usepackage{amsfonts}

\usepackage{color}

\usepackage{paralist}
\usepackage{enumitem}
\setlist{itemsep=1pt, parsep=1pt, leftmargin=15pt}

\usepackage{hyperref}
\usepackage{nameref}

\makeatletter
\def\namedlabel#1#2{\begingroup
   \def\@currentlabel{#2}%
   \label{#1}\endgroup
}
\makeatother

\newtheorem{thm}{Theorem}

\newtheorem{prop}{Proposition}

\newcommand{\ahull}{\mathrm{Aff}}

\newcommand{\rnk}{\mathrm{rank}}
\newcommand{\tr}{\mathrm{Tr}}
\newcommand{\inside}{\mathrm{int}}

\newcommand{\Eps}{\mathcal{E}}

\newcommand{\bbR}{\mathbb{R}}

\newcommand{\conv}{\text{conv}}

\newcommand{\plots}{plots}

\begin{document}

\title{An Extended Frank-Wolfe Method with ``In-Face" Directions, and its Application to Low-Rank Matrix Completion}
\author{Robert M. Freund\thanks{MIT Sloan School of Management, 77 Massachusetts Avenue, Cambridge, MA   02139
%(\href{mailto:pgrigas@mit.edu}{pgrigas@mit.edu}).  This author's research has been partially supported through an NSF Graduate Research Fellowship and the
({mailto:  rfreund@mit.edu}).  This author's research is supported by AFOSR Grant No. FA9550-15-1-0276, the MIT-Chile-Pontificia Universidad Cat\'olica de Chile Seed Fund, and the MIT-Belgium Universit\'{e} Catholique de Louvain Fund.}
\and Paul Grigas\thanks{MIT Operations Research Center, 77 Massachusetts Avenue, Cambridge, MA   02139
%(\href{mailto:rahulmaz@mit.edu}{rahulmaz@mit.edu})}
({mailto:  pgrigas@mit.edu}).  This author's research has been partially supported through an NSF Graduate Research Fellowship, the MIT-Chile-Pontificia Universidad Cat\'olica de Chile Seed Fund, and the MIT-Belgium Universit\'{e} Catholique de Louvain Fund.}
\and Rahul Mazumder\thanks{MIT Sloan School of Management, 77 Massachusetts Avenue, Cambridge, MA   02139
%(\href{mailto:rfreund@mit.edu}{rfreund@mit.edu}).  This author's research is supported by AFOSR Grant No. FA9550-11-1-0141 and the MIT-Chile-Pontificia Universidad
({mailto:  rahulmaz@mit.edu}).   This author's research was partially supported by ONR Grant N000141512342 and a grant from the Moore Sloan Foundation.}}
%\date{} % Activate to display a given date or no date (if empty),
         % otherwise the current date is printed

\maketitle

\begin{abstract}
Motivated principally by the low-rank matrix completion problem, we present an extension of the Frank-Wolfe method that is designed to induce near-optimal solutions on low-dimensional faces of the feasible region.  This is accomplished by a new approach to generating ``in-face" directions at each iteration, as well as through new choice rules for selecting between in-face and ``regular" Frank-Wolfe steps. Our framework for generating in-face directions generalizes the notion of away-steps introduced by Wolfe.  In particular, the in-face directions always keep the next iterate within the minimal face containing the current iterate.  We present computational guarantees for the new method that trade off efficiency in computing near-optimal solutions with upper bounds on the dimension of minimal faces of iterates. We apply the new method to the matrix completion problem, where low-dimensional faces correspond to low-rank matrices. We present computational results that demonstrate the effectiveness of our methodological approach at producing nearly-optimal solutions of very low rank.  On both artificial and real datasets, we demonstrate significant speed-ups in computing very low-rank nearly-optimal solutions as compared to either the Frank-Wolfe method or its traditional away-step variant.
\end{abstract}

%\begin{abstract} We present an extension of the Frank-Wolfe method that is designed to induce near-optimal solutions on low-dimensional faces of the feasible region.  This is accomplished by a new approach to generating directions at each iteration as well as through new choice rules for selection of directions at each iteration.  We present computational guarantees for the method that trade off efficiency in computing near-optimal solutions with upper bounds on the dimension of minimal faces of iterates.  We apply our method to the low-rank matrix completion problem.  We present computational results for large-scale low-rank matrix completion problems that demonstrate its effectiveness.  We demonstrate order-of-magnitude speed-up in computing low-rank near-optimal solutions.\end{abstract}

\section{Introduction}
In the last ten years the problem of \emph{matrix completion} (see, for example, \cite{rennie2005fast, candes2009exact, candes2010power}) has emerged as an important and ubiquitous problem in statistics and machine learning, with applications in diverse areas \cite{candes2010matrix,troyanskaya2001missing}, with perhaps the most notable being recommender systems \cite{bennett2007netflix, bell2007lessons, koren2009matrix}.  In matrix completion one is given a partially observed data matrix $X\in \mathbb{R}^{m \times n}$, i.e., there is only knowledge of the entries $X_{ij}$ for $(i,j) \in \Omega$ where $\Omega \subseteq \{1, \ldots, m\} \times \{1, \ldots, n\}$ (often, $|\Omega| \ll m \times n$), and the task is to predict (fill in) the unobserved entries of $X$. The observed entries are possibly contaminated with noise, i.e., $X = Z^\ast + \Eps$ where $Z^\ast \in \bbR^{m \times n}$ represents the ``true data matrix" and $\Eps$ is the noise term, and the goal is to accurately estimate the entire matrix $Z^\ast$, which most importantly includes estimating the entries $Z^\ast_{ij}$ for $(i,j) \not \in \Omega$. Clearly, this problem is in general ill-posed -- without any restrictions, the unobserved entries can take on any real values. The ill-posed nature of the problem necessitates that any successful approach must, either explicitly or implicitly, make some type of assumption(s) about underlying \emph{structure} of the matrix $Z^\ast$. The most common approach, especially without {\it a priori} knowledge about the data-generating mechanism, is to assume that the matrix $Z^\ast$ is \emph{low-rank}. This situation is similar to the ``bet on sparsity" principle in linear regression \cite{ESLBook}:  if $Z^\ast$ does not have low-rank structure, then we cannot expect any method to successfully fill in the missing entries; on the other hand, if $Z^\ast$ does have low-rank, then a method that makes such a structural assumption should have a better chance at success.

%Matrix completion has applications in a number of diverse areas (\textcolor{red}{Paul/Rahul breifly expand}), perhaps most notably in the area of recommender systems. A popular example is the \emph{Netflix Prize} \cite{bennett2007netflix, bell2007lessons, koren2009matrix} where rows correspond to users, columns correspond to movies, and entries $X_{ij}$ corresponds to cardinal ratings that users give to movies they have watched. The primary interest of Netflix in this context is to give customized movie recommendations to each user, which can of course be accomplished by using the given ratings data to generate predictions for how each user would rate their unseen movies. The low-rank assumption is quite natural here -- we can model movies as combinations of a few genres and/or model users as combinations of of a few demographic types, all with noise, leading to a model for ratings that is based only on the affinities between movies/users and these more general types. A remarkable feature of low-rank models is that this categorization exercise does not need to be directly performed {\it a priori}; rather, it is an implicit part of the low-rank assumption.

The low-rank structural assumption naturally leads to the following optimization problem:
\begin{equation}\label{hard}\begin{array}{llrl}
P_r: & z^* \ \ := \ \  & \min\limits_{Z \in \mathbb{R}^{m \times n}} & \displaystyle\frac{1}{2}\sum_{(i,j) \in \Omega} (Z_{ij}-X_{ij})^2 \\ \\
& & \mbox{s.t.} & \mbox{rank}(Z) \le r \ ,
\end{array}\end{equation}
where $r$ is a parameter representing the assumed belief about the rank of $Z^\ast$.
%Under certain situations, for example in problems without noise, it may be more appropriate to ``flip" \eqref{hard} so that the data-fidelity terms appear in the constraints (perhaps constrained even further such that $Z_{ij} = X_{ij}$ for $(i,j) \in \Omega$) and the rank term appears in the objective. For our purposes however, we are primarily interested in formulations akin to \eqref{hard}.
Notice that \eqref{hard} is a combinatorially hard problem due to the rank constraint \cite{chistov1984complexity}.

Pioneered by \cite{fazel2002matrix}, a promising strategy for attacking \eqref{hard} is to use the \emph{nuclear norm} as a proxy for the rank. Recall that, for a given $Z \in \mathbb{R}^{m \times n}$, the sum of the singular values of $Z$ is a norm often referred to as the nuclear norm. Directly replacing the combinatorially hard rank constraint in \eqref{hard} with a constraint on the nuclear norm of $Z$ leads to the following convex optimization problem:
\begin{equation}\label{easy}
\begin{array}{llrl}
NN_\delta: & f^* \ \ := \ \  & \min\limits_{Z \in \mathbb{R}^{m \times n}} & \displaystyle\frac{1}{2}\sum_{(i,j) \in \Omega} (Z_{ij}-X_{ij})^2 \\ \\
& & \mbox{s.t.} & \|Z\|_{N1} \le \delta \ .
\end{array}
\end{equation}
Let ${\cal B}_{N1}(Z,\delta):=\{ Y \in \mathbb{R}^{m \times n}: \|Y-Z\|_{N1} \le \delta\}$ denote the nuclear norm ball of radius $\delta$ centered at the point $Z$, so that the feasible region of \eqref{easy} is ${\cal B}_{N1}(0,\delta)$. Despite its apparent absence from the problem formulation, it is nevertheless imperative that computed solutions of \eqref{easy} have low rank.  Such low-rank computed solutions are coerced by the nuclear norm constraint, and there has been substantial and influential work showing that, for many types of data generating mechanisms, an optimal solution of \eqref{easy} will have appropriately low rank (see, for instance, \cite{fazel2002matrix, candes2009exact, candes2010power, recht2010guaranteed}).  This line of work typically focuses on studying the properties of optimal solutions of \eqref{easy}, and thus abstracts away the choice of algorithm to solve \eqref{easy}. Although this abstraction may be reasonable in some situations, and is certainly a reasonable way to study the benefits of nuclear norm regularization, it may also be limiting. Indeed, in recent years, the notion that ``convex optimization is a black box" has become increasingly unreasonable. Concurrently with the explosion of ``big data" applications, there has been a substantial amount of recent work on the development and analysis of algorithms for huge-scale convex optimization problems where interior point methods and other polynomial-time algorithms are ineffective. Moreover, there has been an increasing interest in algorithms that directly promote desirable structural properties of their iterates.  One such algorithm that satisfies both of these properties -- scalability to huge-size problems and structurally favorable iterates -- is the Frank-Wolfe method and its extensions, which is the starting point of the work herein.  Indeed, much of the recent computational work for matrix completion is based on directly applying first-order methods and related methods that have structurally favorable iterates. \cite{ji2009accelerated, cai2010singular, ma2011fixed}. Mazumder et al. \cite{mazumder2010spectral} develop a related algorithm based on SVD soft thresholding that efficiently utilizes the special structure of matrix completion problems. In one of the earlier works applying the Frank-Wolfe method to nuclear norm regularized problems, Jaggi and Sulovsk \cite{jaggi2010simple} consider first lifting the nuclear norm regularized problem \eqref{easy} to a problem over the semidefinite cone and then apply the Frank-Wolfe method. Harchaoui et al. \cite{harchaoui2012conditional} pointed out that the Frank-Wolfe method can be applied directly to the nuclear norm regularized problem \eqref{easy}, and also developed a variant of the method that applies to penalized nuclear norm problems. Mishra et al. \cite{mishra2013low} develop a second-order trust region method that shares a few curious similarities with the extended Frank-Wolfe method developed herein. Mu et al. \cite{mu2014scalable} consider a hybrid proximal gradient/Frank-Wolfe method for low-rank matrix and tensor recovery. Rao et al. \cite{rsw} consider a variant of Frank-Wolfe with ``backward steps" (which differ from away steps) in the general context of atomic norm regularization.

\noindent {\bf The Frank-Wolfe method, in-face directions, and structural implications.}  Due to its low iteration cost and convenient structural properties (as we shall soon discuss), the Frank-Wolfe method (also called the conditional gradient method) is especially applicable in several areas of machine learning and has thus received much renewed interest in recent years, see \cite{jaggi2013revisiting, harchaoui, GF-FW, lan2013complexity} and the references therein. The Frank-Wolfe method, originally developed by \cite{frank-wolfe} in the context of quadratic programming, was later generalized to convex optimization problems with smooth (differentiable) convex objective functions and bounded convex feasible regions, of which \eqref{easy} is a particular instance. Indeed, letting $f(Z) := \tfrac{1}{2}\sum_{(i,j) \in \Omega} (Z_{ij}-X_{ij})^2$ denote the least squares objective in \eqref{easy}, it is easy to see that $f(\cdot)$ is a smooth convex function, and the feasible region of \eqref{easy} is ${\cal B}_{N1}(0,\delta)$, which is a bounded convex set.

As applied to problem \eqref{easy}, the Frank-Wolfe method proceeds at the current iterate $Z^k$ by solving a linear optimization subproblem to compute $\tilde Z^k \gets \arg\min_{Z \in {\cal B}_{N1}(0,\delta)} \left\{\nabla f(Z^k) \bullet Z\right\}$ (here ``$\bullet$'' denotes the usual trace inner product) and updates the next iterate as
\begin{equation}\label{update}
Z^{k+1} \gets Z^{k} + \bar{\alpha}_k(\tilde{Z}^k - Z^k)  \ ,
\end{equation}
for some $\bar\alpha_k \in [0, 1]$. It can be shown, see for instance \cite{GF-FW, jaggi2013revisiting, harchaoui} and as we expand upon in Section \ref{fw}, that for appropriate choices of the step-size sequence $\{\bar\alpha_k\}$ it holds that \begin{equation}\label{tradeoff}
f(Z^k) - f^\ast \leq \frac{8\delta^2}{k+3} \ \ \ \ \text{and} \ \ \ \ \rnk(Z^k) \leq k + 1 \ .
\end{equation}
The bound on the objective function gap in \eqref{tradeoff} is well understood and follows from a standard analysis of the Frank-Wolfe method. The bound on the rank of $Z^k$ in \eqref{tradeoff}, while also well understood, follows from the special structure of the nuclear norm ball. Specifically, and as we further expand upon in Sections \ref{fw} and \ref{comp}, for the nuclear norm regularized matrix completion problem \eqref{easy}, the solutions to the linear optimization subproblem solved at each iteration are specially structured -- they are rank one matrices arising from the leading left and right singular vectors of the matrix $\nabla f(Z^k)$. Thus, assuming that $Z^0$ is a rank one matrix, the simple additive form of the updates \eqref{update} leads to the bound on the rank in \eqref{tradeoff}.
The above bound on the rank of $Z^k$ is precisely the ``favorable structural property" of the iterates of the Frank-Wolfe method that was mentioned to earlier, and when combined with the bound on the objective function gap in \eqref{tradeoff} yields a nice tradeoff between data fidelity and low-rank structure.

Here we see that in the case of the Frank-Wolfe method, the \emph{properties of the algorithm} provide additional insight into how problem \eqref{easy} induces low-rank structure.  A natural question is: can the tradeoff given by \eqref{tradeoff} be improved, either theoretically or practically or both? That is, can we modify the Frank-Wolfe method in a way that maintains the bound on the objective function gap in \eqref{tradeoff} while strictly improving the bound on the rank? This is the motivation for the development of what we call ``in-face" directions and their subsequent analysis herein. We define an in-face direction to be any descent direction that keeps the next iterate within the minimal face of ${\cal B}_{N1}(0,\delta)$ containing the current iterate. It turns out that the faces of the nuclear norm ball are characterized by the (thin) SVDs of the matrices contained within them \cite{So} -- thus, an in-face direction will move to a new point $Z^{k+1}$ with a similar SVD structure as $Z^k$, and moreover will keep the rank of  $Z^{k+1}$ the same (or will lower it, even better), i.e., $\rnk(Z^{k+1}) \le \rnk(Z^k)$. Clearly if we can find good in-face directions, then the bound on the rank in \eqref{tradeoff} will be improved. At the same time, if there are no in-face directions that are ``good enough" with respect to improvements in objective function values, then a ``regular" Frank-Wolfe direction may be chosen, which will usually increase the rank of the next iterate by one. In this paper, we develop an extension of the Frank-Wolfe method that incorporates in-face directions and we provide both a precise theoretical analysis of the resulting tradeoff akin to \eqref{tradeoff}, as well as computational results that demonstrate significant improvements over existing methods both in terms of ranks and run times.
\subsection{Organization/Results}

The paper is organized as follows. In Section \ref{fw}, after reviewing the basic Frank-Wolfe Method and the away-step modification of Wolfe and Gu\'{e}lat and Marcotte, we present our extended Frank-Wolfe Method based on ``in-face'' directions (in addition to regular Frank-Wolfe directions), this being the main methodological contribution of the paper.  This In-Face Extended Frank-Wolfe Method is specifically designed to induce iterates that lie on low-dimensional faces of the feasible set $S$, since low-dimensional faces of the feasible region contain desirable ``well-structured" points (sparse solutions when $S$ is the $\ell_1$ ball, low-rank matrices when $S$ is the nuclear norm ball).  The in-face directions are any directions that keep the current iterate in its current minimal face of $S$.  We present two main strategies for computing in-face directions: \emph{(i)} away-steps as introduced by Wolfe \cite{wolfe} and Gu\'{e}lat and Marcotte \cite{gm}, and \emph{(ii)} approximate full optimization of the objective $f(\cdot)$ over the current minimal face. The In-Face Extended Frank-Wolfe Method uses a simple decision criteria for selecting between in-face and regular Frank-Wolfe directions.  In Theorem \ref{RRguarantee} we present computational guarantees for the In-Face Extended Frank-Wolfe Method.  These guarantees essentially show that the In-Face Extended Frank-Wolfe Method maintains $O(c/k)$ convergence after $k$ iterations (which is optimal for Frank-Wolfe type methods in the absence of polyhedral structure or strong convexity \cite{lan2013complexity}), all the while promoting low-rank iterates via the parameters of the method which affect the constant $c$ above, see Theorem \ref{RRguarantee} for specific details.

In Section \ref{comp} we discuss in detail how to apply the In-Face Extended Frank-Wolfe Method to solve the matrix completion problem \eqref{easy}.  We resolve issues such as characterizing and working with the minimal faces of the nuclear norm ball, solving linear optimization subproblems on the nuclear norm ball and its faces, computing steps to the boundary of the nuclear norm ball, and updating the SVD of the iterates.  In Proposition \ref{rank-prop} we present a bound on the ranks of the matrix iterates of the In-Face Extended Frank-Wolfe Method that specifies how the in-face directions reduce the rank of the iterates over the course of the algorithm.  Furthermore, as a consequence of our developments we also demonstrate, for the first time, how to effectively apply the away-step method of \cite{gm} to problem \eqref{easy} in a manner that works with the natural parameterization of variables $Z \in \bbR^{m \times n}$ (as opposed to an ``atomic" form of \cite{gm}, as we expand upon at the end of Section \ref{sec:away-steps}).

Section \ref{compcomp} contains a detailed computational evaluation of the In-Face Extended Frank-Wolfe Method and discusses several versions of the method based on different strategies for computing in-face directions and different algorithmic parameter settings. We compare these versions to the regular Frank-Wolfe method, the away-step method of \cite{gm}, and an atomic version of \cite{gm} (as studied in \cite{simon, beck2015linearly, pena2015neumann}) on simulated problem instances as well as on the MovieLens10M dataset.  Our results demonstrate that the In-Face Extended Frank-Wolfe Method (in different versions) shows significant computational advantages in terms of delivering low rank and low run time to compute a target optimality gap.  Especially for larger instances, one version of our method delivers very low rank solutions with reasonable run times, while another version delivers the best run times, beating existing methods by a factor of $10$ or more.

%In Section \ref{fw}, we review the Frank-Wolfe method, basic convergence guarantees, and the away-step variant introduced by Wolfe.  We then present our extended Frank-Wolfe method based on in-face directions, we discuss some strategies for computing different types of in-face directions, and we derive computational guarantees that bound: \emph{(i)} the objective function gap after $k$ iterations, and \emph{(ii)} the dimension of the minimal face containing the $k^{\text{th}}$ iterate.  In Section \ref{comp}, we discuss how these results translate to the specific setting of nuclear norm regularized matrix completion \eqref{easy}, where low-dimensional faces correspond to low-rank matrices, and we discuss computational challenges and implementation issues.  Finally, in Section \ref{compcomp} we present a detailed computational evaluation of the new extended Frank-Wolfe method applied to the matrix completion problem. \textcolor{red}{More to come as this is finalized}

\subsection{Notation}

Let $E$ be a finite-dimensional linear space.  For a norm $\|\cdot\|$ on $E$, let $\|\cdot\|^*$ be the associated dual norm, namely $\|c\|^{\ast} := \max \{ c^Tz : \|z\| \le 1\}$ and $c^Tz$ denotes the value of the linear operator $c$ acting on $z$.  The ball of radius $\delta$ centered at $\bar z$ is denoted ${\cal B}(\bar z, \delta):=\{z : \|z-\bar z\| \le \delta \}$.  For $X, Y \in \mathbb{S}^{k \times k}$ (the set of $k \times k$ symmetric matrices), we write ``$X\succeq 0$'' to denote that $X$ is symmetric and positive semidefinite, ``$X\succeq Y$'' to denote that $X-Y\succeq 0$, and ``$X\succ 0$'' to denote that $X$ is positive definite, etc.  For a given $Z \in \mathbb{R}^{m \times n}$ with $r:= \rnk(Z)$, the (thin) singular value decomposition (SVD) of $Z$ is $Z = UDV^T$ where $U \in \mathbb{R}^{m \times r}$ and $V \in \mathbb{R}^{n \times r}$ are each orthonormal ($U^TU=I$ and $V^TV=I$), and $D=\mbox{Diag}(\sigma_1, \dots, \sigma_r)$ comprises the non-zero (hence positive) singular values of $Z$. The nuclear norm of $Z$ is then defined to be $\|Z\|_{N1} := \sum_{j=1}^r \sigma_j$ .   (In much of the literature, this norm is denoted $\|\cdot\|_*$; we prefer to limit the use of ``$*$'' to dual norms, and hence we use the notation $\|\cdot\|_{N1}$ instead.) Let ${\cal B}_{N1}(Z,\delta):=\{ Y \in \mathbb{R}^{m \times n}: \|Y-Z\|_{N1} \le \delta\}$ denote the nuclear norm ball of radius $\delta$ centered at the point $Z$. Let $\|Z\|_F$ denote the Frobenius norm of $Z$, namely $\|Z\|_F = \sqrt{\sum_{j=1}^r \sigma_j^2} = \sqrt{\tr(Z^TZ)}$.  The dual norm of the nuclear norm is the largest singular value of a matrix and is denoted by $\|\cdot\|^*_{N1}=\|\cdot\|_{N\infty}$; given $S \in \mathbb{R}^{m \times n}$ with SVD $S=UDV^T$, then $\|S\|_{N\infty}=\max\{\sigma_1, \ldots, \sigma_r \}$ where $D=\mbox{Diag}(\sigma_1, \dots, \sigma_r)$.  A spectrahedron is a set of the form ${\cal S}_t^k := \{X \in \mathbb{S}^{k\times k} : X \succeq 0, \ I \bullet X \le t  \}$ or $\bar{\cal S}_t^k := \{X \in \mathbb{S}^{k\times k}  : X \succeq 0, \ I \bullet X = t  \}$, where ``$\bullet$'' denotes the usual trace inner product.

\section{Frank-Wolfe Method, Away Steps, and In-Face Steps}\label{fw}

Problem \eqref{easy} is an instance of the more general problem:

\begin{equation}\label{poi3}
f^* \ \ := \ \ \min_{x \in S} \ f(x)
\end{equation}

%\begin{equation}\label{poi3}\begin{array}{llrl}& f^* \ \ := \ \  & \mini_x & f(x) \\ \\
%& & \mbox{s.t.} & x \in S \ , \end{array}\end{equation}

\noindent where $S \subset E$ is a closed and bounded convex set, and $f(\cdot)$ is a differentiable convex function on $S$.  We first review solving instances of \eqref{poi3} using the Frank-Wolfe method, whose basic description is given in
Algorithm~\ref{algo-vfw}.

\begin{algorithm}
\caption{Frank-Wolfe Method for optimization problem \eqref{poi3}}\label{algo-vfw}
\begin{algorithmic}
\STATE Initialize at $x_0 \in S$, (optional) initial lower bound $B_{-1}$, $k \gets 0$ .\\
$ \ $ \\
At iteration $k$:
\STATE 1. Compute $\nabla f (x_k)$ .
\STATE 2. Compute $\tilde x_k \gets \arg\min\limits_{x \in S}\{f(x_k) + \nabla f (x_k)^T(x - x_k)\}$ .\\
\ \ \ \ \ \ \ \ \ \ $B^w_k \leftarrow f(x_k) + \nabla f (x_k)^T(\tilde x_k - x_k)$ .\\
\ \ \ \ \ \ \ \ \ \ Update best bound: $B_k \leftarrow \max\{B_{k-1}, B^w_k\}$ .
\STATE 3. Set $x_{k+1} \gets x_k + \bar{\alpha}_k(\tilde x_k - x_k)$, where $\bar{\alpha}_k \in [0,1]$ .
\end{algorithmic}
\end{algorithm}

\noindent Typically the main computational burden at each iteration of the Frank-Wolfe method is solving the linear optimization subproblem in Step (2.) of Algorithm \ref{algo-vfw}. The quantities $B^w_k$ are lower bounds on the optimal objective function value $f^*$ of \eqref{poi3}, a fact which follows easily from the gradient inequality, see Jaggi \cite{jaggi2013revisiting} or \cite{GF-FW}, and hence $B_k = \max\{B_{-1}, B^w_0, \ldots, B^w_k\}$ is also a lower bound on $f^*$.  The lower bound sequence $\{B_k\}$ can be used in a variety of step-size strategies \cite{GF-FW} in addition to being useful in termination criteria.

\noindent When the step-size sequence $\{\bar \alpha_k\}$ is chosen using the simple rule $\bar\alpha_k := \frac{2}{k+2}$, then the Frank-Wolfe method has the following computational guarantee at the $k^{\mathrm{th}}$ iteration, for $k \ge 0$:
\begin{equation}\label{nemguarantee} f(x_k) - B_{k} \le f(x_k) - f^* \le \frac{2LD^2}{k+3} \ , \end{equation}
\noindent where $D:= \max_{x, y \in S} \|x-y\|$ is the diameter of $S$, and $L$ is a Lipschitz constant of the gradient of $f(\cdot)$ on $S$, namely:
\begin{equation}\label{lipschitz} \|\nabla f(x) - \nabla f(y) \|_* \le L \|x-y\| \ \ \ \mbox{for~all~} x,y \in S \ . \end{equation}
If $\bar\alpha_k$ is instead chosen by exact line-search, namely $\bar\alpha_k \gets \arg\min_{\alpha \in [0,1]} f(x_k + \alpha(\tilde x_k -x_k))$, then the guarantee \eqref{nemguarantee} still holds, see Section 3.4 of \cite{GF-FW}, this being particularly relevant when $f(\cdot)$ is a convex quadratic as in \eqref{easy} in which case the exact line-search reduces to a simple formula.  Alternatively, one can consider a step-size rule based on minimizing an upper-approximation of $f(\cdot)$ inherent from the smoothness of the gradient, namely:
\begin{equation}\label{lipschitz2} f(y) \le f(x) + \nabla f(x)^T(y-x) + \tfrac{L}{2}\|y-x\|^2 \ \ \ \mbox{for~all~} x,y \in S \ , \end{equation}
which follows from \eqref{lipschitz} (see \cite{GF-FW}, for example, for a concise proof).  The following is a modest extension of the original analysis of Frank and Wolfe in \cite{frank-wolfe}.

\begin{thm}\label{QAtheorem} {\bf (extension of \cite{frank-wolfe})} Let $\bar L \ge L$ be given, and consider using either an exact line-search or the following step-size rule for the Frank-Wolfe method:

\begin{equation}\label{btrule}\bar\alpha_k \gets \min\left\{\frac{\nabla f (x_k)^T(x_k - \tilde x_k)}{\bar L\|x_k - \tilde x_k\|^2}, 1 \right\} \ \mbox{for~all~} k \ge 0 \ .  \end{equation}
Then $f(x_k)$ is monotone decreasing in $k$, and it holds that:
\begin{equation}\label{btguarantee} f(x_k) - f^* \le f(x_k) - B_k \le \frac{1}{\frac{1}{f(x_0) - B_0} + \frac{k}{2\bar L D^2}} \ \  < \ \ \frac{2\bar L D^2}{k} \ .  \end{equation}
\end{thm}

\noindent {\bf Proof:}
The first inequality of \eqref{btguarantee} follows from the fact that $B_k \ge f^*$, and the third inequality follows from the fact that $f(x^0) \ge f^* \ge B_0 $.  The second inequality can be rewritten as:
$$\frac{1}{f(x_{k}) - B_k} \ge \frac{1}{f(x_0) - B_0} + \frac{k}{2\bar L D^2}  \ , $$ which states that the reciprocal of the optimality bound gap grows at least according to the indicated linear function in $k$.  The above inequality holds trivially for $k=0$, and hence to prove the second inequality of \eqref{btguarantee} it suffices to show that:
\begin{equation}\label{snrc}\frac{1}{f(x_{k+1}) - B_{k+1}} \ge \frac{1}{f(x_k) - B_k} + \frac{1}{2\bar L D^2}  \ \ \ \ \mathrm{for~all~} k \ge 0 \ , \end{equation}
whose proof is given in Appendix \ref{swapoff}, and wherein the monotonicity of $f(x_k)$ is also proved. \qed

In addition to being the crux of the proof of \eqref{btguarantee}, we will also use inequality \eqref{snrc} and related inequalities as the basis for choosing among candidate directions in the in-face extension of Frank-Wolfe that we will develop in Section \ref{qwerty}.

\subsection{Away Steps}\label{sec:away-steps}
In \cite{wolfe} Wolfe introduced the concept of an ``away step'' in a modified version of the Frank-Wolfe method, and Gu\'{e}lat and Marcotte \cite{gm} provided a modification thereof and an extensive treatment of the convergence properties of the away-step-modified Frank-Wolfe method, including eventual linear convergence of the method when the objective function is strongly convex, the feasible region is polyhedral, and a form of strict complementarity holds.  Quite recently there has been much renewed interest in the Frank-Wolfe method with away steps, with most of the focus being on demonstrating global linear convergence with computational guarantees for a particular ``atomic" version of \cite{gm}, see Lacoste-Julien and Jaggi \cite{simon}, Beck and Shtern \cite{beck2015linearly}, and Pe\~{n}a et al. \cite{pena2015neumann}.

Algorithm \ref{GMalgorithm} presents the modified Frank-Wolfe method with away steps as developed in \cite{gm}.  The algorithm needs to work with the minimal face of a point $x \in S$, which is the smallest face of $S$ that contains the point $x$; here we use the notation ${\cal F}_S(x)$ to denote the minimal face of $S$ which contains $x$.  Step (2.) of the modified Frank-Wolfe method is the ``away step'' computation, where $\check x_k$ is the point on the current minimal face ${\cal F}_S(x_k)$ that is farthest along the ray from the ``bad'' solution $\hat x_k$ through the current point $x_k$.  Step (3.) of the modified method is the regular Frank-Wolfe step computation, which is called a ``toward step'' in \cite{gm}.  Notice that implementation of the away-step modified Frank-Wolfe method depends on the ability to characterize and work with the minimal face ${\cal F}_S(x_k)$ of the iterate $x^k$. When $S$ is not a polytope this minimal face capability is very much dependent on problem-specific knowledge of the structure of the set $S$.

%In \cite{wolfe} Wolfe introduced the concept of an ``away step'' in a modified version of the Frank-Wolfe method, and an extensive treatment of the convergence properties of the away-step-modified Frank-Wolfe method is presented in Gu\'{e}lat and Marcotte \cite{gm}.  Algorithm \ref{GMalgorithm} presents the modified Frank-Wolfe method with away steps as developed in \cite{gm}.  Unlike the Frank-Wolfe method, Algorithm \ref{GMalgorithm} is \emph{not} affine invariant in the general sense of Proposition \ref{affine} (it is, however, invariant under invertible affine transformations). Indeed, Algorithm \ref{GMalgorithm} requires the ability to work with the minimal face containing any given point, which is highly dependent on the specifically chosen parameterization of variables -- two different parameterizations of the same point may lead to nonequivalent minimal faces containing that point. Formally, we fix the specific parameterization to be that defined by problem \eqref{poi3}, and we denote ${\cal F}_S(x)$ as the minimal face of $S$ that contains the point $x$.

\begin{algorithm}
\caption{Modified Frank-Wolfe Method with Away Steps, for optimization problem \eqref{poi3}}\label{GMalgorithm}
\begin{algorithmic}
\STATE Initialize at $x_0 \in S$, (optional) initial lower bound $B_{-1}$, $k \gets 0$ .\\
$ \ $ \\
At iteration $k$:
\STATE 1. Compute $\nabla f (x_k)$ .
\STATE 2. Compute $\hat x_k \gets  \arg\max\limits_{x}\{\nabla f (x_k)^Tx : x \in {\cal F}_S(x_k)\}$ .\\
\ \ \ \ \ \ \ \ \ \ $\alpha_k^{\mathrm{stop}} \gets \arg\max\limits_{\alpha}\{\alpha : x_k + \alpha(x_k - \hat x_k) \in {\cal F}_S(x_k)\}$ .\\
\ \ \ \ \ \ \ \ \ \ $\check x_k \gets x_k + \alpha_k^{\mathrm{stop}}(x_k - \hat x_k)$ . \\
\STATE 3. Compute $\tilde x_k \gets \arg\min\limits_{x}\{\nabla f (x_k)^Tx : x \in S\}$ .\\
\ \ \ \ \ \ \ \ \ \ $B^w_k \leftarrow f(x_k) + \nabla f (x_k)^T(\tilde x_k - x_k)$ .\\
\ \ \ \ \ \ \ \ \ \ Update best bound: $B_k \leftarrow \max\{B_{k-1}, B^w_k\}$ .
\STATE 4. Choose descent direction:  \\
\ \ \ \ \ \ \ \ \ \ If $\nabla f(x_k)^T (\tilde x_k - x_k) \le \nabla f(x_k)^T (x_k - \hat x_k)$, then $d_k \gets \tilde x_k - x_k$ and $\bar\beta_k \gets 1$ ; \\
\ \ \ \ \ \ \ \ \ \ Else $d_k \gets x_k - \hat x_k$ and $\bar\beta_k \gets \alpha_k^{\mathrm{stop}}$ .

\STATE 5. Set $x_{k+1} \gets x_k + \bar{\alpha}_k d_k$, where $\bar{\alpha}_k \in [0,\bar\beta_k]$ .
\end{algorithmic}
\end{algorithm}

\noindent The convergence of the modified Frank-Wolfe method is proved in Theorem 4 of \cite{gm} under the assumption that $\bar\alpha_k$ in  Step (5.) is chosen by exact line-search; however a careful review of the proof therein shows that convergence is still valid if one uses a step-size rule in the spirit of \eqref{btrule} that uses the quadratic upper-approximation of $f(\cdot)$ using $L$ or $\bar L \ge L$.  The criterion in Step (4.) of Algorithm \ref{GMalgorithm} for choosing between the regular Frank-Wolfe step and the away step seems to tailor-made for the convergence proof in \cite{gm}.  In examining the proof of convergence in \cite{gm}, one finds the fact that $\hat x_k$ is an extreme point is not relevant for the proof, nor even is the property that $\hat x_k$ is a solution of a linear optimization problem.  Indeed, this begs for a different way to think about both generating and analyzing away steps, which we will do shortly in Subsection \ref{qwerty}.

\noindent {\bf Away-steps are {\em not} affine-invariant.}  The feasible region $S$ of \eqref{poi3} can always be (implicitly) expressed as $S = \conv(\mathcal{A})$ where $\mathcal{A}$ is a collection of points in $S$ that includes all of the extreme points of $S$.  In fact, in many current applications of Frank-Wolfe and its relatives, $S$ is explicitly constructed as $S := \conv(\mathcal{A})$ for a given collection $\mathcal{A}$ whose members are referred to as  ``atoms"; and each atom $\tilde x \in \mathcal{A}$ is a particularly ``simple" point (such as a unit coordinate vector $\pm e^i$, a rank-$1$ matrix, etc.). Let us consider the (possibly infinite-dimensional) vector space $V := \{\alpha \in \mathbb{R}^{\mathcal{A}} : \text{support}(\alpha) \text{ is finite}\}$, and define the simplicial set $\Delta_{\mathcal{A}}$ by:
$$ \Delta_{\mathcal{A}} := \left\{\alpha \in V : \alpha \ge 0, \ \sum_{\tilde x \in \mathcal{A}} \alpha_{\tilde x} = 1\right\} \ , $$
and consider the linear map $M(\cdot) : \Delta_{\mathcal{A}} \to S$ such that $M(\alpha) := \sum_{\tilde x \in \mathcal{A}} \alpha_{\tilde x}\tilde x$.  Then it is obvious that the following two optimization problems are equivalent:
\begin{equation}\label{big-simplex}
\begin{array}{rcl}
 \min\limits_{x \in S} f(x) \ \ & \equiv & \ \ \min\limits_{\alpha \in \Delta_{\mathcal{A}}} f(M(\alpha)) \ ,
\end{array}\end{equation} where the left-side is our original given problem of interest \eqref{poi3} and the right-side is its re-expression using the convex weights $\alpha \in \Delta_{\mathcal{A}}$ as the variables.  Furthermore, it follows from the fundamental affine-invariance of the regular Frank-Wolfe method (Algorithm \ref{algo-vfw}) as articulated by Jaggi \cite{jaggi2013revisiting} that the Frank-Wolfe method applied to the left-side problem above is equivalent (via the linear mapping $M(\cdot)$) to the Frank-Wolfe method applied to the right-side problem above.  However, this affine invariance property does {\em not} extend to the away-step modification of the method, due to the fact that the facial structure of a convex set is not affine invariant -- not even so in the case when $S$ is a polytope.  This is illustrated in Figure \ref{fig:atomic}.  The left panel shows a polytopal feasible region $S \subset \mathbb{R}^3$ with ${\cal F}_S(x_k)$ highlighted.  The polytope $S$ has $10$ extreme points.  The right panel shows ${\cal F}_S(x_k)$ by itself in detail, wherein we see that $x_k = .25 \tilde x_1 + .25 \tilde x_2 + .50 \tilde x_3$ (among several other combinations of other extreme points of ${\cal F}_S(x_k)$ as well).  Let us now consider the atomic expression of the set $S$ using the $10$ extreme points $S$ and instead expressing our problem in the format of the right-side of \eqref{big-simplex}, wherein the feasible region is the unit simplex in $\mathbb{R}^{10}$, namely $\Delta_{10}:= \{\alpha \in \mathbb{R}^{10} : \alpha \ge 0, \ e^T\alpha = 1\}$ where $e=(1, \ldots, 1)$ is the vector of ones.  If the current iterate $x_k$ is given the atomic expression $\alpha_k = (.25,.25,.50, 0 , 0,0,0,0,0,0)$, then the minimal face ${\cal F}_{\Delta_{10}}(\alpha_k)$ of $\alpha_k$ in $\Delta_{10}$ is the sub-simplex $\{\alpha \in \mathbb{R}^{10} : \alpha \ge 0, \ e^T\alpha = 1, \ \alpha_4= \cdots \alpha_{10} = 0\}$, which corresponds back in $S \subset \mathbb{R}^3$ to the narrow triangle in the right panel of Figure \ref{fig:atomic}, and which is a small subset of the pentagon corresponding to the minimal face ${\cal F}_S(x_k)$ of $x_k$ in $S$.  Indeed, this example illustrates the general fact that the faces of the atomic expression of $S$ will always correspond to subsets of the faces of the facial structure of $S$.  Therefore, away-step sub-problem optimization computations using the original representation of $S$ will optimize over larger subsets of $S$ than will the corresponding computations using the atomic re-expression of the problem.  Indeed, we will show in Section \ref{compcomp} in the context of matrix completion that by working with the original representation of the set $S$ in the setting of using away-steps, one can obtain significant computational savings over working with the atomic representation of the problem.

Last of all, we point out that the away-step modified Frank-Wolfe methods studied by Lacoste-Julien and Jaggi  \cite{simon}, Beck and Shtern \cite{beck2015linearly}, and Pe\~{n}a et al. \cite{pena2015neumann} can all be viewed as applying the away-step method (Algorithm \ref{GMalgorithm}) to the ``atomic'' representation of the optimization problem as in the right-side of \eqref{big-simplex}.
\begin{figure}[h!]
\centering
\scalebox{.8}{\begin{tabular}{c c c}
%\multicolumn{3}{c}{Comparison of Away-step vs.\ Atomic-away-step Directions}\\
\includegraphics[width=0.38\textwidth,height=0.38\textheight, trim= 0mm -5mm 0mm 0mm, clip=true, keepaspectratio=true]{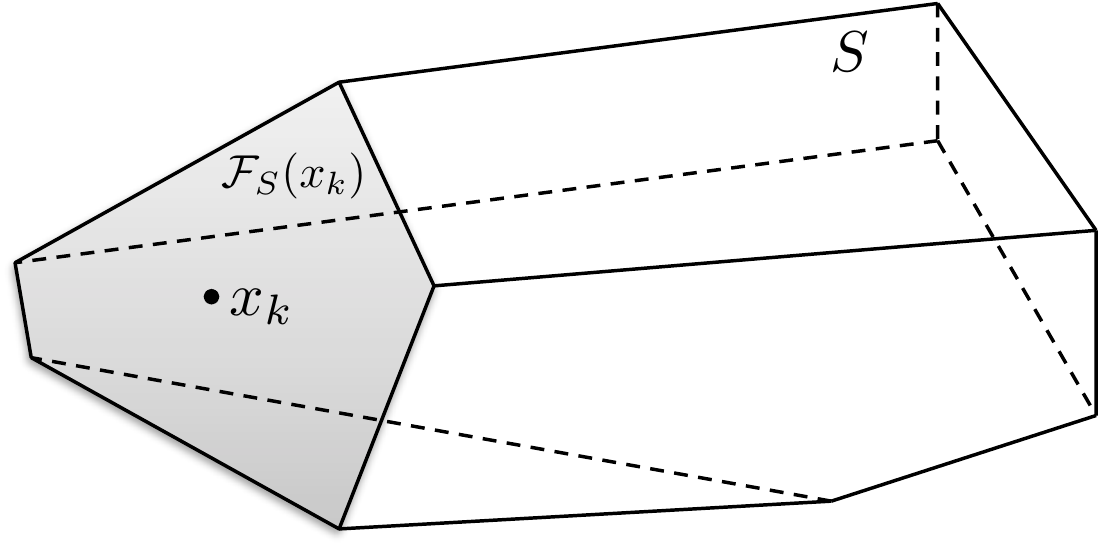}
&
\hspace{10mm} &
\includegraphics[width=0.3\textwidth,height=0.3\textheight, trim= 0mm 0mm 0mm 0mm, clip=true, keepaspectratio=true]{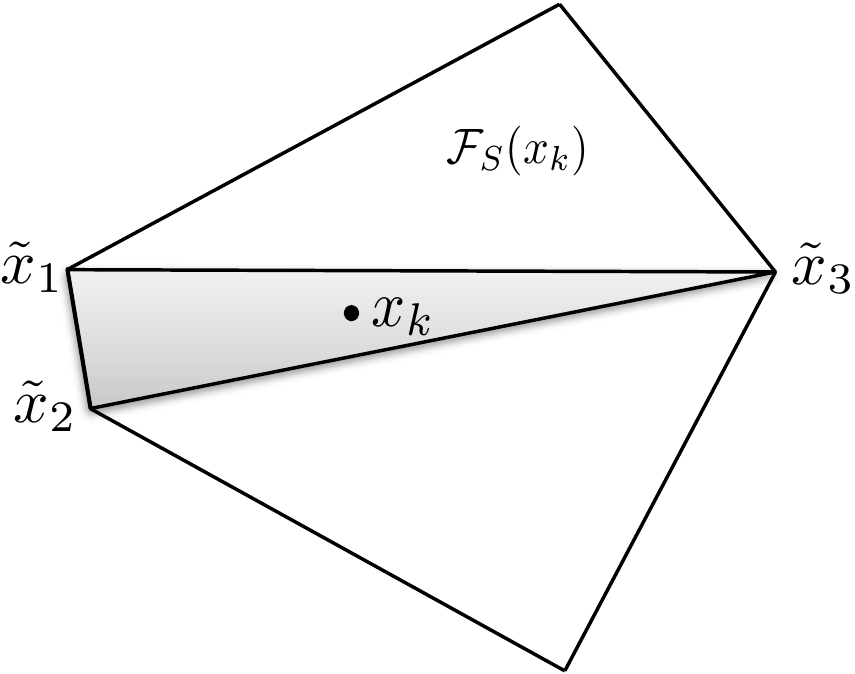}
\end{tabular}}
\caption{Illustration that facial structure of a polytope is not affine invariant.}\label{fig:atomic}
\end{figure}

\subsection{An ``In-Face'' Extended Frank-Wolfe Method}\label{qwerty}

\noindent Here we present an ``in-face'' extension of the Frank-Wolfe method, that is significantly more general than the away-step method of Wolfe \cite {wolfe} and Gu\'{e}lat and Marcotte \cite{gm} (Algorithm \ref{GMalgorithm}), and its atomic version studied by by Lacoste-Julien and Jaggi \cite{simon}, Beck and Shtern \cite{beck2015linearly}, and Pe\~{n}a et al. \cite{pena2015neumann}.  The method is motivated by the desire to compute and work with points $x$ that have specific structure, usually sparsity (in the case when $x$ is a vector or matrix) or low-rank (in the case when $x$ is a matrix).   More generally, we will think of the structure as being related to the dimension of the minimal face ${\cal F}_S(x)$ of $S$ containing $x$.  The algorithm is designed to balance progress towards two different goals, namely (i) progress towards optimizing the objective function, and (ii) the aim of having the iterates lie in low-dimensional faces of $S$.  In the case of the matrix completion problem \eqref{easy} in particular, if an iterate lies in a low-dimensional face of $S$ then the iterate will have low rank (see Theorem \ref{facial}).  Such low rank is advantageous not only because we want the output solution to have low rank, but also because a low-rank iterate yields a substantial reduction in the computation costs at that iteration.  This last point will be further developed and exploited in Sections \ref{comp} and \ref{compcomp}.

We present our ``In-Face Extended Frank-Wolfe Method'' in Algorithm \ref{RRa}.  At Step (2.) of each iteration the algorithm works with an ``in-face'' direction $d_k$ which will keep the next candidate point in the current minimal face ${\cal F}_S(x_k)$.  This is equivalent to requiring that $x_k + d_k$ lies in the affine hull of ${\cal F}_S(x_k)$, which is denoted by $\ahull({\cal F}_S(x_k))$.  Other than the affine hull condition, the direction $d_k$ can be {\em any} descent direction of $f(\cdot)$ at $x_k$ if such a direction exists.  The candidate iterate $x^B_k$ is generated by stepping in the direction $d_k$ all the way to the relative boundary of the minimal face of the current point $x_k$.  The point  $x^A_k$ is the candidate iterate generated using the in-face direction and a suitable step-size $\bar \beta_k$, perhaps chosen by exact line-search or by a quadratic approximation rule.  In Steps (3a.) and (3b.) the algorithm applies criteria for choosing which, if any, of $x^B_k$ or $x^A_k$ to accept as the next iterate of the method.  If the criteria are not met for either $x^B_k$ or $x^A_k$, then the method computes a regular Frank-Wolfe step in Step (3c.) and updates the lower bound $B_k$.

\begin{algorithm}
\caption{In-Face Extended Frank-Wolfe Method for optimization problem \eqref{poi3}} \label{RRa}
\begin{algorithmic}
\STATE Initialize at $x_0 \in S$, (optional) initial lower bound $B_{-1}$, $k \gets 0$ .\\
Choose $\bar L \ge L$, $\bar D \ge D$, and constants $\gamma_1$, $\gamma_2$ satisfying $0 \le \gamma_1 \le \gamma_2 \le 1$ .\\
$ \ $ \\
At iteration $k$:
\STATE 1. Compute $\nabla f (x_k)$ .  $B_k \gets B_{k-1}$ .
\STATE 2. Compute direction $d_k$ for which $x_k + d_k \in \ahull({\cal F}_S(x_k))$ and $\nabla f(x_k)^Td_k < 0$.  (If no $d_k$ exists, go to Step (3c.).)\\
\ \ \ \ \ \ \ \ \ \ $\alpha_k^{\mathrm{stop}} \gets \arg\max\limits_{\alpha}\{\alpha : x_k + \alpha d_k \in {\cal F}_S(x_k)\}$ .\\
\ \ \ \ \ \ \ \ \ \ $x^B_k := x_k + \alpha_k^{\mathrm{stop}} d_k$ . \\
\ \ \ \ \ \ \ \ \ \ $x^A_k := x_k + \bar\beta_k d_k$ where $\bar\beta_k \in [0,\alpha_k^{\mathrm{stop}}]$ . \\
\STATE 3. Choose next iterate:  \\
\ \ \ \ \ \ \ \ \ \ (a.) (Go to a lower-dimensional face.) \\
\ \ \ \ \ \ \ \ \ \ \ \ \ \ \ \ \ \ \ \ If $ \frac{1}{f(x^B_k) - B_k } \ge \frac{1}{f(x_k) - B_k } + \frac{\gamma_1}{2\bar L \bar D^2} \ , $ set $x_{k+1} \gets x_k^B$. \\
\ \ \ \ \ \ \ \ \ \ (b.) (Stay in current face.) \\
\ \ \ \ \ \ \ \ \ \ \ \ \ \ \ \ \ \ \ \ Else if $ \frac{1}{f(x^A_k) - B_k } \ge \frac{1}{f(x_k) - B_k } + \frac{\gamma_2}{2\bar L \bar D^2} \ , $ set $x_{k+1} \gets x_k^A$.\\
\ \ \ \ \ \ \ \ \ \ (c.) (Do regular FW step and update lower bound.)  Else, compute:  \\
\ \ \ \ \ \ \ \ \ \ \ \ \ \ \ \ \ $\tilde x_k \gets \arg\min\limits_{x}\{\nabla f (x_k)^Tx : x \in S\}$ . \\
\ \ \ \ \ \ \ \ \ \ \ \ \ \ \ \ \ $x_{k+1} \gets x_k + \bar\alpha_k (\tilde x_k - x_k)$ where $\bar\alpha_k \in [0,1]$ . \\
\ \ \ \ \ \ \ \ \ \ \ \ \ \ \ \ \ $B^w_k \leftarrow f(x_k) + \nabla f (x_k)^T(\tilde x_k - x_k)$, \ $B_k \leftarrow \max\{B_{k-1}, B^w_k\}$ .

\end{algorithmic}
\end{algorithm}

Let us now discuss a few strategies for computing in-face directions. One recovers the away-step direction of the method of Gu\'{e}lat and Marcotte \cite{gm} by choosing:
\begin{equation}\label{gm-strategy}
d_k \gets x_k - \hat x_k \ , \ \ \ \text{where} \ \ \ \hat x_k \gets  \arg\max\limits_{x}\{\nabla f (x_k)^Tx : x \in {\cal F}_S(x_k)\} \ .
\end{equation}
Another natural way to compute a suitable $d_k$, that is computationally facile for relatively low-dimensional faces and for certain problem instances (including matrix completion), is to directly solve for an (approximately) optimal objective function solution over the low-dimensional face ${\cal F}_S(x_k)$ and thereby set:
\begin{equation}\label{mazumder}
d_k \gets x^M_k - x_k \ , \ \ \ \text{where} \ \ \ x^M_k \gets \arg\min_x \{f(x) : x \in {\cal F}_S(x_k) \} \ .
\end{equation}
Note that in this case, we may naturally set $\bar\beta_k := 1$. Another related type of in-face direction that may be of interest is to consider a regular Frank-Wolfe step within ${\cal F}_S(x_k)$, whereby we select:
\begin{equation}\label{toward-strategy}
d_k \gets \tilde x_k^F - x_k \ , \ \ \ \text{where} \ \ \ \tilde x_k^F \gets \arg\min\limits_{x}\{\nabla f (x_k)^Tx : x \in {\cal F}_S(x_k)\} \ .
\end{equation}
One may interpret this ``in-face Frank-Wolfe step" as a single iteration of the Frank-Wolfe method applied to the subproblem in \eqref{mazumder}. As we elaborate in Section \ref{compcomp} when discussing the practical merits of these approaches, our main interests are in the away-step strategy \eqref{gm-strategy} and the full optimization strategy \eqref{mazumder}.  Both of these in-face Frank-Wolfe step strategies lead to significant computational advantages over the regular Frank-Wolfe method, as will be shown in Section \ref{compcomp}.

One immediate advantage of the In-Face Extended Frank-Wolfe Method (Algorithm \ref{RRa}) compared to the away-step modified Frank-Wolfe method of Gu\'{e}lat and Marcotte \cite{gm} (Algorithm \ref{GMalgorithm}) has to do with the number and sizes of linear optimization sub-problems that are solved.  Algorithm \ref{GMalgorithm} needs to solve two linear optimization subproblems at each iteration -- a ``small'' subproblem  on the minimal face ${\cal F}_S(x_k)$ and a ``large'' subproblem on the entire set $S$.  In contrast, even when computing directions using away-step computations, algorithm \ref{RRa} must solve the ``small'' linear optimization problem on the minimal face ${\cal F}_S(x_k)$, but the method will only need to solve the ``large'' subproblem on the entire set $S$ if it needs to process Step (3c.).  The computational advantage from not having to solve the ``large'' subproblem at every iteration will be shown in Section \ref{compcomp}.

We now discuss the criteria that are used in Step (3.) to choose between the next step $x^B_k$ that lies in the relative boundary of the current minimal face ${\cal F}_S(x_k)$, the step $x^A_k$ that does not necessarily lie in the relative boundary of the current minimal face ${\cal F}_S(x_k)$, and a regular Frank-Wolfe step.  We see from Step (3.) of Algorithm \ref{RRa} that a regular Frank-Wolfe step will be chosen as the next iterate unless the criteria of either Step (3a.) or (3b.) are met.  The criteria in Step (3a.) is met if $x^B_k$ (which lies on the relative boundary of ${\cal F}_S(x_k)$ by virtue of the definition of $\alpha_k^{\mathrm{stop}}$) provides sufficient decrease in the optimality gap as measured with the criterion:
$$\frac{1}{f(x^B_k) - B_k } \ge \frac{1}{f(x_k) - B_k } + \frac{\gamma_1}{2\bar L \bar D^2} \ . $$ The criteria in Step (3b.) is met if $x^A_k$ provides sufficient decrease in the optimality gap as measured similar to above but using $\gamma_2$ rather than $\gamma_1$.  Since $\gamma_1 \le \gamma_2$, Step (3a.) requires a lesser decrease in the optimality bound gap than does Step (3b.).

In settings when we strongly desire to compute iterates that lie on low-dimensional faces (as in the low-rank matrix completion problem \eqref{easy}), we would like the criteria in Steps (3a.) and (3b.) to be relatively easily satisfied (perhaps with it being even easier to satisfy the criteria in Step (3a.) as this will reduce the dimension of the minimal face).  This can be accomplished by setting the values of $\gamma_1$ and $\gamma_2$ to be lower rather than higher.  Indeed, setting $\gamma_1=0$ ensures in Step (3a.) that the next iterate lies in a lower-dimensional face whenever $x^B_k$ (which by definition lies in a lower dimensional face than $x_k$) does not have a worse objective function value than $f(x_k)$.  Also, if one sets $\gamma_2$ to be smaller, then the criteria in Step (3b.) is more easily satisfied, which ensures that the new iterate will remain in the current face ${\cal F}_S(x_k)$ as desired when the criterion of Step (3b.) is satisfied.

As we have discussed, the ability to induce solutions on low-dimensional faces by setting $\gamma_1$ and $\gamma_2$ to have low values can be extremely beneficial.  However, this all comes at a price in terms of computational guarantees, as we now develop.  Before presenting the computational guarantee for Algorithm \ref{RRa} we first briefly discuss step-sizes; the step-size $\bar\beta_k$ for steps to the in-face point $x_k^A$ are determined in Step (2.), and the step-size $\bar\alpha_k$ for regular Frank-Wolfe steps is chosen in Step (3c.).  One strategy is to choose these step-sizes using an exact line-search if the line-search computation is not particularly burdensome (such as when $f(\cdot)$ is a quadratic function).  Another strategy is to determine the step-sizes according to the quadratic upper approximation of $f(\cdot)$ much as in Theorem \ref{QAtheorem}, which in this context means choosing the step-sizes as follows:
\begin{equation}\label{letat}\bar\beta_k := \min\left\{\frac{-\nabla f(x_k)^T d_k}{\bar L \|d_k\|^2} , \  \alpha_k^{\mathrm{stop}}\right\} \  \  \ , \ \ \ \bar\alpha_k := \min\left\{\frac{\nabla f (x_k)^T(x_k - \tilde x_k)}{\bar L\|x_k - \tilde x_k\|^2}, \  1 \right\} \ .
\end{equation}
Let $N^a_k$, $N^b_k$, and $N^c_k$ denote the number of times within the first $k$ iterations that the iterates are chosen according to the criteria in Steps (3a.), (3b.), and (3c.), respectively.  Then $k = N^a_k + N^b_k + N^c_k$, and we have the following computational guarantee.

\begin{thm}\label{RRguarantee} Suppose that the step-sizes used in Algorithm \ref{RRa} are determined either by exact line-search or by \eqref{letat}.  After $k$ iterations of Algorithm \ref{RRa} it holds that:
$$f(x_k) - f^* \le f(x_k) - B_k \le \frac{1}{\displaystyle\frac{1}{f(x_0) - B_0} + \frac{\gamma_1 N^a_k}{2\bar L \bar D^2} + \frac{\gamma_2 N^b_k}{2\bar L \bar D^2} + \frac{N^c_k}{2\bar L \bar D^2} } < \frac{2\bar L \bar D^2}{\gamma_1 N^a_k + \gamma_2 N^b_k + N^c_k} \ . $$
\end{thm}

\noindent {\bf Proof:}  The first inequality is true since $B_k \le f^*$, and the third inequality is true since $f(x_0) \ge B_0$, so we need only prove the second inequality, which can be equivalently written as:
\begin{equation}\label{leominster}\frac{1}{f(x_k) - B_k} \ge \ \displaystyle\frac{1}{f(x_0) - B_0} + \frac{\gamma_1 N^a_k}{2\bar L \bar D^2} + \frac{\gamma_2 N^b_k}{2\bar L \bar D^2} + \frac{N^c_k}{2\bar L \bar D^2}  \ .\end{equation} Notice that \eqref{leominster} is trivially true for $k=0$ since $N^a_k = N^b_k = N^c_k = 0$ for $k=0$.
Let $\Delta^k :=(f(x_k) - B_k)^{-1}$ denote the inverse objective function bound gap at iteration $k$.  Then if the next iterate is chosen by satisfying the criteria in Step (3a.), it holds that $\Delta^{k+1} \ge (f(x_{k+1}) - B_{k})^{-1} \ge \Delta^k  + \frac{\gamma_1}{2\bar L \bar D^2}$ where the first inequality derives from $B_{k+1} \ge B_k$ and the second inequality is from the criterion of Step (3a.).  Similarly, if the next iterate is chosen by satisfying the criteria in Step (3b.), it holds using similar logic that $\Delta^{k+1} \ge  \Delta^k  + \frac{\gamma_2}{2\bar L \bar D^2}$.   And if the next iterate is chosen in Step (3c.), namely we take a regular Frank-Wolfe step, then inequality \eqref{snrc} holds, which is $\Delta^{k+1} \ge  \Delta^k  + \frac{1}{2\bar L \bar D^2}$.  Applying induction then establishes \eqref{leominster}, which completes the proof. \qed

Here we see that choosing smaller values of $\gamma_1$ and $\gamma_2$ can have a detrimental effect on the progress of the algorithm in terms of the objective function optimality gap, while larger values ensure better convergence guarantees.  At the same time, smaller values of $\gamma_1$ and $\gamma_2$ are more effective at promoting iterates to lie on low-dimensional faces.  Thus there is a clear tradeoff between objective function optimality gap accuracy and low-dimensional structure, dictated by the values of $\gamma_1$ and $\gamma_2$. One strategy that is worth studying is setting $\gamma_1 = 0$ and $\gamma_2 $ to be relatively large, say $\gamma_2=1$ for example.  With these values of the parameters we take an in-face step in Step (3a.) (which lowers the dimension of the face of the iterate) whenever doing so will not adversely affect the objective function value. This and other strategies for setting $\gamma_1$ and $\gamma_2$ will be examined in Section \ref{compcomp}.

%\textcolor{red}{Our extended Frank-Wolfe method differs from the away-step Frank-Wolfe method in two important ways:  \emph{(i)} we can consider many other in-face directions besides away-steps, and \emph{(ii)} our decision criteria for selecting between in-face and regular Frank-Wolfe directions is very different and more versatile than the inner-product rule used in the away-step variant of the Frank-Wolfe method.  We also note that the linear convergence analysis for most away-step methods does not apply to the matrix completion problem \eqref{easy} because the feasible region of ${\cal B}_{N1}(0,\delta)$ of \eqref{easy} is not polyhedral.  Fundamentally, for matrix completion applications, we are primarily interested in the ability of in-face directions to promote structured, i.e., low-rank, iterates, as opposed to improve convergence rates.}

%\textcolor{red}{The ``in-face'' extension of the Frank-Wolfe method that we develop herein is most closely related to the method of \cite{gm} -- indeed as we discuss in Section \ref{qwerty}, it is one of several mechanisms that that can be used to generate in-face directions.}

\noindent {\bf A simplified algorithm in the case of full optimization over the current minimal face}.  Let us further examine the dynamics of Algorithm \ref{RRa} in the case of \eqref{mazumder}, where we select the in-face direction by fully optimizing the objective function $f(\cdot)$ over the low-dimensional face ${\cal F}_S(x_k)$. Consider performing an in-face step in this case, i.e., suppose that the next iterate is chosen according to the criteria in Steps (3a.)/(3b.) (recall that we set $\bar\beta_k := 1$ in this case). Then, at the next iteration, Algorithm \ref{RRa} is guaranteed to select a regular Frank-Wolfe step via Step (3c.). Indeed, since the next iterate $x_{k+1}$ is chosen as the optimal solution over ${\cal F}_S(x_k)$, by definition there are no descent directions at $x_{k+1}$ that remain within ${\cal F}_S(x_{k+1}) \subseteq {\cal F}_S(x_{k})$ and thus no valid in-face directions to be selected. Here we see that the parameters $\gamma_1$ and $\gamma_2$ are superfluous -- a much more natural procedure is to simply alternate between regular Frank-Wolfe steps and fully optimizing over ${\cal F}_S(x_k)$. This bears some similarity to, but is distinct from, the ``fully corrective" variant of Frank-Wolfe, see, e.g., \cite{jaggi2013revisiting, harchaoui}. (Indeed, these two algorithms coincide if we consider this alternating procedure applied to the lifted problem \eqref{big-simplex}.) In this case, the following computational guarantee follows simply from Theorem \ref{QAtheorem}.

\begin{prop}\label{prop-alt}
Consider a slight variation of Algorithm \ref{RRa} that alternates between full optimizations \eqref{mazumder} over the current face ${\cal F}_S(x_k)$ and regular Frank-Wolfe steps, with step-size $\bar\alpha_k$ chosen either by exact line-search or by a quadratic approximation rule \eqref{letat}. For simplicity, consider one iteration to consist of both of these operations in sequence. Then, for all $k \geq 0$, it holds that:
\begin{equation*}
f(x_k) - f^* \le f(x_k) - B_k \le \frac{1}{\frac{1}{f(x_0) - B_0} + \frac{k}{2\bar L D^2}} \ \  < \ \ \frac{2\bar L D^2}{k} \ .
\end{equation*}
\end{prop}

\section{Solving Matrix Completion Problems using the In-Face Extended Frank-Wolfe Method}\label{comp}

We now turn our attention to solving instances of \eqref{easy} using the the In-Face Extended Frank-Wolfe Method (Algorithm \ref{RRa}).  We work directly with the natural parameterization of variables as $m \times n$ matrices $Z \in \bbR^{m \times n}$ (although, as we discuss in Section \ref{svd-updating}, we utilize low-rank SVD updating to maintain the variables in an extremely memory efficient manner). Recall that the objective function of \eqref{easy} is $f(Z) := \tfrac{1}{2}\sum_{(i,j) \in \Omega} (Z_{ij}-X_{ij})^2$, whose gradient is $\nabla f(Z) = (Z-X)_\Omega$.  The feasible region of \eqref{easy} is $S={\cal B}_{N1}(0,\delta)$, which notation we shorten to ${\cal B} := {\cal B}_{N1}(0,\delta)$. We first discuss the specification and implementation issues in using Algorithm \ref{RRa} to solve \eqref{easy}.
%\subsection{Norm, Lipschitz constant, and Diameter}\label{pneum}

We will fix the norm on $Z$ to be the nuclear norm $\|\cdot\|_{N1}$, whose dual norm is easily seen to be $\|\cdot\|_{N1}^* = \|\cdot\|_{N\infty}$.  Then it is plain to see that under the nuclear norm it holds that the Lipschitz constant of the objective function of \eqref{easy} is $L=1$. This follows since for any $Z,Y \in \mathbb{R}^{m \times n}$ we have:
\begin{align*}
\|\nabla f(Z) - \nabla f(Y)\|_{N\infty} ~&\leq~ \|\nabla f(Z) - \nabla f(Y)\|_{N2} ~=~ \|(Z-X)_\Omega-(Y-X)_\Omega \|_{F} \\
& \\
&\leq~ \|(Z-Y) \|_{F} ~=~ \|(Z-Y) \|_{N2} ~\leq~ \|(Z-Y) \|_{N1} \ .
\end{align*}
%$$\begin{array}{rcl}  \|\nabla f(Z) - \nabla f(Y)\|_{N\infty} &\le& \|\nabla f(Z) - \nabla f(Y)\|_{N2} \\ &=&\|(Z-X)_\Omega-(Y-X)_\Omega \|_{F}  \\  &\le& \|(Z-Y) \|_{F}\\ &=& \|(Z-Y) \|_{N2} \\ &\le& \|(Z-Y) \|_{N1} \ . \end{array}$$
\noindent Since the feasible region of \eqref{easy} is $S={\cal B} := {\cal B}_{N1}(0,\delta)$ it follows that the diameter of $S$ is $D = 2\delta$.  Let us use the superscript $Z^k$ to denote the $k^{\mathrm{th}}$ iterate of the algorithm, to avoid confusion with the subscript notation $Z_{ij}$ for indices of the $(i,j)^{\mathrm{th}}$ component of $Z$.

\subsection{Characterization of faces of the nuclear norm ball}\label{faces}

In order to implement Algorithm \ref{RRa} we need to characterize and work with the minimal face of ${\cal B} = {\cal B}_{N1}(0,\delta)$ containing a given point.  Let  $\bar Z \in {\cal B}$ be given.  The minimal face of ${\cal B}$ containing $\bar Z$ is formally notated as ${\cal F}_{{\cal B}}(\bar Z)$.  We have the following characterization of ${\cal F}_{{\cal B}}(\bar Z)$ due to So \cite{So}:

\begin{thm}\label{facial} {\bf (So \cite{So})} Let $\bar Z \in {\cal B}$ have thin SVD $\bar Z = UDV^T$ and let $r=\rnk(\bar Z)$.  Let ${\cal F}_{{\cal B}}(\bar Z)$ denote the minimal face of ${\cal B}$ containing $\bar Z$.  If $\sum_{j=1}^r \sigma_j = \delta$, then $\bar Z \in \partial {\cal B}$ and it holds that:
$${\cal F}_{{\cal B}}(\bar Z) = \{ Z \in \mathbb{R}^{m\times n} : Z = UMV^T \ \mathrm{for~some~} M \in \mathbb{S}^{r\times r}, \ M \succeq 0, \ I \bullet M = \delta \} \ , $$
and $\dim({\cal F}_{{\cal B}}(\bar Z)) = r(r+1)/2 - 1$. Otherwise $\sum_{j=1}^r \sigma_j < \delta$ and it holds that ${\cal F}_{{\cal B}}(\bar Z) = {\cal B}$ and $\dim({\cal F}_{{\cal B}}(\bar Z)) = \dim(\mathcal{B}) = m \times n$. \qed
\end{thm}

Theorem \ref{facial} above is a reformulation of Theorem 3 of So \cite{So}, as the latter pertains to square matrices ($m=n$) and also does not explicitly treat the minimal faces containing a given point, but is a trivial extension of So's theorem.

Theorem \ref{facial} explicitly characterizes the correspondence between the faces of the nuclear norm ball and low-rank matrices on its boundary. Note from Theorem \ref{facial} that if $\bar Z \in \partial{\cal B}$ and $r=\rnk(\bar Z)$, then ${\cal F}_{{\cal B}}(\bar Z)$ is a linear transformation of the $r \times r$ spectrahedron $\bar{\mathcal{S}}_\delta^r := \{M \in \mathbb{S}^{r\times r} : M \succeq 0, \ I \bullet M = \delta  \}$.  This property will be most useful as it will make it very easy to compute in-face directions, especially when $r$ is relatively small, as we will see in Section \ref{alex} and Section \ref{sowhat}.

\subsection{Linear optimization subproblem solution for regular Frank-Wolfe step}\label{linear_opt}

In Step (3c.) of Algorithm \ref{RRa} we need solve a linear optimization problem.  Here we show how this can be done efficiently.  We need to compute:
\begin{equation}\label{easy2}
\tilde{Z}^{k}  ~\gets~ \arg\min\limits_{Z \in {\cal B}_{N1}(0,\delta)} \ \nabla f(Z^k) \bullet Z \ .
\end{equation}
Then an optimal solution $\tilde Z^k$  is readily seen to be:
\begin{equation}\label{rank1}\tilde{Z}^{k} \gets - \delta \mathbf{u}_{k}\mathbf{v}^T_{k}\end{equation}
where
$\mathbf{u}_{k}$ and $\mathbf{v}_{k}$ denote the left and right singular vectors, respectively, of the matrix
$\nabla f(Z^k)$ corresponding to the largest singular value of $\nabla f(Z^k)$.  Therefore computing $\tilde{Z}^{k}$ in Step (3c.) is relatively easy so long as the computation of the largest singular value of $\nabla f(Z^k)$ and associated left and right eigenvalues thereof are easy to accurately compute. If $|\Omega|$ is relatively small, then there are practically efficient methods (such as power iterations) that can effectively leverage the sparsity of $\nabla f(Z^k)$.

\subsection{Strategies and computation of the in-face direction $D^k$}\label{alex}

Let $D^k$ denote the in-face direction computed in Step (2.) of Algorithm \ref{RRa}.  As suggested in Section \ref{qwerty}, we present and discuss two different strategies for generating a suitable $D^k$, namely (i) using an away-step approach \eqref{gm-strategy}, and (ii) directly solving for an optimal objective function solution over the low-dimensional face ${\cal F}_{\cal B}(Z^k)$ \eqref{mazumder}. In either case, computing $D^k$ requires working with the thin SVD of $Z^k$, which characterizes ${\cal F}_{\cal B}(Z^k)$ as stated in Theorem \ref{facial}. Of course, the thin SVD of $Z^k$ can be recomputed at every iteration, but this is generally very inefficient. As we expand upon in Section \ref{svd-updating}, the thin SVD of $Z^{k+1}$ can be efficiently \emph{updated} from the thin SVD of $Z^k$ by utilizing the structure of the regular Frank-Wolfe and in-face directions. For now, we simply assume that we have access to the thin SVD of $Z^k$ at the start of iteration $k$.

\noindent {\bf Away-step Strategy.}  Here we choose $D^k$ by setting $D^k \gets Z^k - \hat Z^k$ where $\hat Z^k$ is the solution of the linear optimization maximization problem over the current minimal face, as in Step (2.) of the away-step algorithm (Algorithm \ref{GMalgorithm}).  We compute the ``away-step point'' $\hat Z^k$ by computing:
\begin{equation}\label{pindyck}
\hat Z^k ~\gets~ \arg\max_{Z \in {\cal F}_{\cal B}(Z^k)} \ \nabla f(Z^k) \bullet Z \ ,
\end{equation}
and set $D^k \gets Z^k - \hat Z^k$.  To see how to solve \eqref{pindyck} efficiently, we consider two cases, namely when $Z^k \in \inside({\cal B})$ and when $Z^k \in \partial({\cal B})$.  In the case when $Z^k \in \inside({\cal B})$, then ${\cal F}_{\cal B}(Z^k) = {\cal B}$ and the optimal solution in \eqref{pindyck} is just the negative of the solution of \eqref{rank1}, namely $\hat Z^k = \delta \mathbf{u}_{k}\mathbf{v}^T_{k}$.

In the case when $Z^k \in \partial({\cal B})$, $\rnk(Z^k) = r$, and $Z^k$ has thin SVD $Z^k = UDV^T$, we use the characterization of ${\cal F}_{\cal B}(Z^k)$ in Theorem \ref{faces} to reformulate \eqref{pindyck} as:
\begin{equation}\label{poi5}
\hat Z^k ~\gets~ U\hat M^kV^T \ \ \ \mathrm{where} \ \ \ \hat M^k ~\gets~ \arg\max\limits_{M \in \bar{\mathcal{S}}_\delta^r} \ G^k \bullet M \ ,
\end{equation}
and where $G^k := \tfrac{1}{2}(V^T\nabla f(Z^k)^TU + U^T\nabla f(Z^k)V)$ so that $\nabla f(Z^k) \bullet UMV^T = G^k \bullet M$ for all $M \in \bar{\mathcal{S}}_\delta^r$.
An optimal solution to the subproblem in \eqref{poi5} is readily seen to be
\begin{equation}\label{rank11}\hat{M}^k \gets \delta \mathbf{u_k}\mathbf{u_k}^T\end{equation}
where $\mathbf{u}_{k}$ is the normalized eigenvector corresponding to the largest eigenvector of the $r \times r$ symmetric matrix $G^k$.  Therefore computing $\hat{Z}^{k}$ in \eqref{pindyck} is relatively easy so long as the computation of the largest eigenvalue of $G^k$ and associated eigenvector thereof are easy to accurately compute.  Furthermore, note that $\hat Z^k = U\hat M^kV^T = \delta U \mathbf{u_k}\mathbf{u_k}^TV^T$ is a rank-one matrix.

The development of the in-face Frank-Wolfe step strategy \eqref{toward-strategy} in this case is quite similar. Indeed, we simply replace the maximization in \eqref{poi5} with a minimization, which corresponds to a smallest eigenvalue computation, and set $D^k$ accordingly.

\noindent {\bf Direct Solution on the Minimal Face.}   In this strategy we use the alternating version of Algorithm \ref{RRa} described at the end of Section \ref{qwerty} and we choose $D^k$ by setting $D^k \gets \bar Z^k - Z^k$ where $\bar Z^k$ optimizes (exactly or perhaps only approximately) the original objective function $f(Z)$ over the current minimal face, under the assumption that such optimization can be done efficiently and accurately. Indeed, when $Z^k \in \inside({\cal B})$, then we default to the previous away-step strategy since optimizing over the minimal face is identical to the original problem \eqref{easy}. Otherwise, when $Z^k = UDV^T \in \partial({\cal B})$ we again use the characterization of ${\cal F}_{\cal B}(Z^k)$ in Theorem \ref{faces} to compute $\bar Z^k$ as:
\begin{equation}\label{berndt}
\bar Z^k ~\gets~ U\bar{M}^kV^T \ \ \ \mathrm{where} \ \ \ \bar M^k ~\gets~ \arg\min\limits_{M \in \bar{\mathcal{S}}_\delta^r} \ f(UMV^T) \ .
\end{equation}
Of course, it is only sensible to consider this strategy when  $Z^k$ has low rank, for otherwise \eqref{berndt} is nearly as difficult to solve as the original problem \eqref{easy} whose solution we seek to approximate using the In-face Extended Frank-Wolfe Method.  Since $f(\cdot)$ is a convex quadratic function, it follows that the subproblem in \eqref{berndt} is solvable as a semidefinite/second-order conic optimization problem and thus conic interior-point methods may be practical. Alternatively, one can approximately solve \eqref{berndt} by taking a number of steps of any suitably effective method, such as a proximal/accelerated first-order method \cite{tseng} (or even the Frank-Wolfe method itself).

%Since $f(\cdot)$ is a convex quadratic function, it follows that the subproblem in \eqref{berndt} is solvable as a conic optimization problem using the semidefinite cone of order $r$ and a single second-order cone constraint to model the quadratic function, where the second-order cone is in $\bbR^{|\Omega| + 1}$.  Therefore, when $r$ is small and $|\Omega|$ is relatively small, this problem is efficiently and accurately solvable in practice using conic interior-point methods such as SDPT3 \cite{sdpt3-userg}.  Alternatively, one could approximately solve \eqref{berndt} by taking a small (or large) number of steps of any suitably effective method, such as an accelerated first-order method \cite{tseng}, or even the Frank-Wolfe method itself, depending on the rank $r$ of $Z^k$ and the relative efficiencies of these methods in this context.  Recall that, for this strategy, we utilize a simplified version of Algorithm \ref{RRa} that alternates between full optimization steps and regular Frank-Wolfe steps as described at the end of Section \ref{qwerty}.

\subsection{Computing the maximal step-size $\alpha_k^{\mathrm{stop}}$ in Step (2.)}\label{sowhat}  Here we describe how to efficiently compute the maximal step-size $\alpha_k^{\mathrm{stop}}$ in Step (2.) of Algorithm \ref{RRa}, which is determined as:
\begin{equation}\label{gibbons} \alpha_k^{\mathrm{stop}} \gets \arg\max\limits_{\alpha}\{\alpha : Z^k + \alpha D^k \in {\cal F}_{\cal B}(Z^k)\} \ . \end{equation}
Let us first assume that $Z^k \in \partial({\cal B})$.  We will utilize the SVD of the current iterate $Z^k = UDV^T$.  Using either the away-step strategy or the direct solution strategy for determining the in-face direction $D^k$ in Section \ref{alex}, it is simple to write $D^k = U\Delta V^T$ for an easily given matrix $\Delta \in \mathbb{S}^{r \times r}$.  Since $Z^k \in \partial({\cal B})$ and $Z^k + D^k \in {\cal F}_{\cal B}(Z^k)$ it holds that $I \bullet D = \delta$ and hence $I \bullet \Delta = 0$.  Using the characterization of ${\cal F}_{\cal B}(Z^k)$ in Theorem \ref{facial} it follows that \eqref{gibbons} can be reformulated as:
\begin{equation}\label{mit}
\alpha_k^{\mathrm{stop}} ~\gets~ \arg\max\limits_{\alpha,M}\{\alpha : UDV^T + \alpha U\Delta V^T = UMV^T, \ M \in \bar{\mathcal{S}}_\delta^r \} ~=~ \arg\max\limits_{\alpha}\{\alpha : D + \alpha \Delta \succeq 0 \}  \ .
\end{equation}
In the case when $D^k$ is chosen using the away-step approach, we have from \eqref{poi5} and \eqref{rank11} that $\Delta := D-\delta{\mathbf{u}_{k}}{\mathbf{u}_{k}}^T$ satisfies $D^k = Z^k - \hat{Z}^k = U\Delta V^T$. In this case the maximum $\alpha$ satisfying \eqref{mit} is easily seen to be $\alpha_k^{\mathrm{stop}} := \left(\delta{\mathbf{u}_{k}}^T D^{-1}{\mathbf{u}_{k}} -1\right)^{-1}$.  When $D^k$ is chosen by some other method, such as the direct solution method on the minimal face, the optimal solution of \eqref{mit} is seen to be $\alpha_k^{\mathrm{stop}} := -\left[\lambda_{\min}\left(D^{-\tfrac{1}{2}}\Delta D^{-\tfrac{1}{2}}\right)\right]^{-1}$.

In the case when $Z^k \in \inside({\cal B})$, then \eqref{gibbons} can be written as $\alpha_k^{\mathrm{stop}} \gets \arg\max\{\alpha : \|Z^k + \alpha D^k \|_{N1} \le \delta \}$, and we use binary search to approximately determine $\alpha_k^{\mathrm{stop}}$.

\subsection{Initial values, step-sizes, and computational guarantees}\label{there}

We initialize Algorithm \ref{RRa} by setting
\begin{equation}\label{Z0}Z^{0} \gets - \delta \mathbf{u}_{0}\mathbf{v}^T_{0}\end{equation}
where
$\mathbf{u}_{0}$ and $\mathbf{v}_{0}$ denote the left and right singular vectors, respectively, of the matrix
$\nabla f(0)$ corresponding to the largest singular value of $\nabla f(0)$.  This initialization corresponds to a ``full step" iteration of the Frank-Wolfe method initialized at $0$ and conveniently satisfies $\rnk(Z^0)=1$ and $Z^0 \in \partial {\cal B}$.
We initialize the lower bound as $B_{-1} \gets \max\left\{f(0) + \nabla f(0) \bullet Z^0, 0\right\}$, where the first term inside the max corresponds to the lower bound generated when computing $Z^0$ and the second term is a valid lower bound since $f^* \ge 0$. Moreover, this initialization has a provably good optimality gap, namely $f(Z^0) \le B_{-1} + 2\delta^2 \leq f^* + 2\delta^2 $, which follows from Proposition 3.1 of \cite{GF-FW}.

Because $f(\cdot)$ is a convex quadratic function, we use an exact line-search to determine $\bar\beta_k$ and $\bar\alpha_k$ in Steps (2.) and (3c.), respectively, since the line-search reduces to a simple formula in this case.

Utilizing the bound on the optimality gap for $Z^0$, and recalling that $L=1$ and $D=2\delta$, we have from Theorem \ref{RRguarantee} that the computational guarantee for Algorithm \ref{RRa} is:
$$f(Z^k) - B_k \le f(Z^k) - f^* \le \frac{1}{\displaystyle\frac{1}{f(Z^0) - B_0} + \frac{\gamma_1 N^a_k}{8\delta^2} + \frac{\gamma_2 N^b_k}{8\delta^2} + \frac{N^c_k}{8\delta^2} } \leq \frac{8\delta^2}{4 + \gamma_1 N^a_k + \gamma_2 N^b_k + N^c_k} \ . $$

\subsection{Efficiently Updating the Thin SVD of $Z^k$}\label{svd-updating}
At each iteration of Algorithm \ref{RRa} we need to access two objects related to the current iterate $Z^k$:  \emph{(i)} the current gradient $\nabla f(Z^k) = (Z^k-X)_\Omega$ (for solving the regular Frank-Wolfe linear optimization subproblem and for computing in-face directions), and \emph{(ii)} the thin SVD $Z^k=UDV^T$ (for computing in-face directions). For large-scale matrix completion problems, it can be very burdensome to store and access all $mn$ entries of the (typically dense) matrix $Z^k$. On the other hand, if $r := \rnk(Z^k)$ is relatively small, then storing the thin SVD of $Z^k$ requires only keeping track of $mr + r + nr$ entries. Thus, when implementing Algorithm \ref{RRa} as discussed above, instead of storing the entire matrix $Z^k$, we store in memory the thin SVD of $Z^k$ (i.e., the matrices $U, V$, and $D$), which we initialize from \eqref{Z0} and efficiently update as follows. Let $D^k$ denote the direction chosen by Algorithm \ref{RRa} at iteration $k \geq 0$, which is appropriately scaled so that $Z^{k+1} = Z^k + D^k$. To compute the thin SVD of $Z^{k+1}$, given the thin SVD of $Z^k$, we consider the cases of regular Frank-Wolfe directions and in-face directions separately. In the case of a regular Frank-Wolfe direction, we have that $D^k = \bar\alpha_k(-\delta \mathbf{u}_{k}\mathbf{v}^T_{k} - Z^k)$ and therefore:
\begin{equation*}
Z^{k+1} = Z^k + \bar\alpha_k(-\delta \mathbf{u}_{k}\mathbf{v}^T_{k} - Z^k) = (1 - \bar\alpha_k)Z^k - \bar\alpha_k\delta \mathbf{u}_{k}\mathbf{v}^T_{k} = (1 - \bar\alpha_k)UDV^T- \bar\alpha_k\delta \mathbf{u}_{k}\mathbf{v}^T_{k} \ .
\end{equation*}
Thus, given the thin SVD of $Z^k$, computing the thin SVD of $Z^{k+1}$ is a scaling plus a rank-$1$ update of the thin SVD, which can be performed very efficiently in terms of both computation time and memory requirements, see \cite{brand2006fast}. An analogous argument applies to the away-step strategy when $Z^k \in \inside({\cal B})$. Otherwise, when $Z^k \in \partial({\cal B})$, recall that we can write any in-face direction as $D^k = U\Delta V^T$ for an easily given matrix $\Delta \in \mathbb{S}^{r \times r}$. Thus we have:
\begin{equation*}
Z^{k+1} = Z^k + D^k = UDV^T + U\Delta V^T = U(D + \Delta)V^T \ .
\end{equation*}
Recall from \eqref{mit} that we have $D + \Delta \succeq 0$. Therefore, to compute the thin SVD of $Z^{k+1}$, we first compute an eigendecomposition of the $r \times r$ symmetric positive semidefinite matrix $D + \Delta$, so that $D + \Delta = RSR^T$ where $R$ is orthonormal and $S$ is diagonal with nonnegative entries, and then update the thin SVD of $Z^{k+1}$ as $Z^{k+1} = (UR)S(VR)^T$.

To compute the current gradient from the thin SVD of $Z^k$, note that $\nabla f(Z^k) = (Z^k-X)_\Omega$ is a sparse matrix that is 0 everywhere except on the $\Omega$ entires; thus computing $\nabla f(Z^k)$ from the thin SVD of $Z^k$ requires performing $|\Omega|$ length $r$ inner product calculations. As compared to storing the entire matrix $Z^k$, our implementation requires a modest amount of extra work to compute $\nabla f(Z^k)$, but the cost of this extra work is far outweighed by the benefits of not storing the entire matrix $Z^k$. Alternatively, it is slightly more efficient to update only the $\Omega$ entries of $Z^k$ at each iteration (separately from the thin SVD of $Z^k$) and to use these entries to compute $\nabla f(Z^k)$.

\subsection{Rank accounting}\label{rankout}

As developed throughout this Section, the computational effort required at iteration $k$ of Algorithm \ref{RRa} depends very much on $\rnk(Z^k)$ for tasks such as computing the in-face direction $D^k$ (using either the away-step approach or direct solution on the minimal face), computing the maximal step-size $\alpha_k^{\mathrm{stop}}$ in Step (3.), and updating the thin SVD of $Z^k$.  Herein we examine how $\rnk(Z^k)$ can change over the course of the algorithm. At any given iteration $i$, there are four relevant possibilities for how the next iterate is chosen:
\begin{itemize}
\item[(a)] The current iterate $Z^i$ lies on the boundary of ${\cal B}$, and the next iterate $Z^{i+1}$ is chosen according to the criteria in Step (3a.).
\item[(b)] The current iterate $Z^i$ lies on the boundary of ${\cal B}$, and the next iterate $Z^{i+1}$ is chosen according to the criteria in Step (3b.).
\item[(c)] The next iterate $Z^{i+1}$ is chosen according to the criteria in Step (3c.).
\item[(d)] The current iterate $Z^i$ lies in the interior of ${\cal B}$, and the next iterate is chosen according to either the criteria in Step (3a.) or Step (3b.).
\end{itemize}
The following proposition presents bounds on rank of $Z^k$.
\begin{prop}\label{rank-prop}
Let $N^a_k$, $N^b_k$, $N^c_k$, and $N^d_k$ denote the number of times within the first $k$ iterations that the above conditions (a), (b), (c), and (d) hold, respectively. Then
\begin{equation}\label{nice-bound}
\rnk(Z^k) ~\le~ k+ 1 - 2N^a_k - N^b_k \ .
\end{equation}
\end{prop}
\begin{proof}
Using the choice of the initial point $Z^0$ developed in Section \ref{there}, it holds that $\rnk(Z^0) = 1$. Now consider the $i^{\mathrm{th}}$ iterate value $Z^i$ for $i = 1, \ldots, k$. If condition (a) holds, then $\rnk(Z^{i+1}) \le \rnk(Z^i) - 1$ since $Z^{i+1}$ lies on a lower-dimensional face of ${\cal F}_{\cal B}(Z^i) \subset {\cal  B}$, whence from Theorem \ref{facial} it follows that $\rnk(Z^{i+1}) \le \rnk(Z^i) - 1$. If instead condition (b) holds, then $\rnk(Z^{i+1}) = \rnk(Z^i) $ since $Z^{i+1}$ lies in the relative interior of ${\cal F}_{\cal B}(Z^i) \subset {\cal B}$. Finally, in either case that condition (c) or condition (d) holds, it follows from \eqref{rank1} that $\tilde Z^i$ is a rank-one matrix and thus it holds that $\rnk(Z^{i+1}) \le \rnk(Z^i) + 1$. Since the four cases above are exhaustive, we have $k = N^a_k + N^b_k + N^c_k + N^d_k$ and we obtain $\rnk(Z^k) \le 1 + N^c_k + N^d_k - N^a_k = k+ 1 - 2N^a_k - N^b_k$.
\end{proof}

\section{Computational Experiments and Results}\label{compcomp}
In this section we present computational results of experiments wherein we apply different versions of the In-Face Extended Frank-Wolfe Method to the nuclear norm regularized matrix completion problem \eqref{easy}.\footnote{All computations were performed using MATLAB R2015b on a 3 GHz Intel Core i7 MacBook Pro laptop.}  Our main focus is on simulated problem instances, but we also present results for the MovieLens10M dataset.  The simulated instances were generated according to the model $X := w_1UV^T + w_2\mathcal{E}$, where the entries of $U \in \bbR^{m \times r}, V \in \bbR^{n \times r}$ and $\mathcal{E} \in \bbR^{m \times n}$ are all i.i.d.\ standard normal random variables, and the scalar parameters $w_1, w_2$ control the signal to noise ratio (SNR), namely $w_1 := 1/\|UV^T\|_F$ and $w_2 := 1/(\text{SNR}\|\mathcal{E}\|_F)$. The set of observed entries $\Omega$ was determined using uniform random sampling of entries with probability $\rho$, where $\rho$ is the target fraction of observed entries.  The objective function $f(\cdot)$ values were normalized so that $f(0) = .5$ and we chose the regularization parameter $\delta$ using a cross-validation-like procedure based on an efficient path algorithm variant of Algorithm \ref{algo-vfw}.\footnote{Specifically, we apply a version of Algorithm \ref{algo-vfw} that periodically increases the value of $\delta$, utilizing the previously found solution as a warm-start at the new value of $\delta$. We maintain a holdout set $\Omega^{\prime}$ and ultimately select the value of $\delta$ that minimizes the least-squares error on this set. %Note that we only require a very low accuracy solution to be found at each value of $\delta$ along the path, whereby our intention is to utilize Algorithm \ref{RRa} to find a higher accuracy solution at the selected value.
}

We study several versions of the In-Face Extended Frank-Wolfe Method (Algorithm \ref{RRa}) based on different strategies for setting the parameters $\gamma_1$,$\gamma_2$, which we compare to the regular Frank-Wolfe method (Algorithm \ref{algo-vfw}) and the away-step method (Algorithm \ref{GMalgorithm}). We also study the atomic version of the away-step method -- which reformulates \eqref{easy} using the right-side of \eqref{big-simplex}.   All methods are implemented according to the details presented in Section \ref{comp}.  We focused on the following eight versions of methods with names given below and where ``\textsc{IF- $\cdot$}'' stands for In-Face:\begin{itemize}
\item[$\bullet$ \namedlabel{algFW}{\textsc{Frank-Wolfe}}] \textsc{Frank-Wolfe} -- Algorithm \ref{algo-vfw}
\item[$\bullet$ \namedlabel{algIFOneOne}{\textsc{IF-(1,1)}}] \textsc{IF-(1,1)} -- Algorithm \ref{RRa} using an away-step strategy, with $\gamma_1 = 1$, $\gamma_2 = 1$
\item[$\bullet$ \namedlabel{algIFZeroOne}{\textsc{IF-(0,1)}}] \textsc{IF-(0,1)} -- Algorithm \ref{RRa} using an away-step strategy,  with $\gamma_1 = 0$, $\gamma_2 = 1$
\item[$\bullet$ \namedlabel{algIFZeroInf}{\textsc{IF-(0,$\infty$)}}] \textsc{IF-(0,$\infty$)} -- Algorithm \ref{RRa} using an away-step strategy, with $\gamma_1 = 0$, $\gamma_2 = \infty$. This corresponds to always moving to the relative boundary of the minimal face containing $Z_k$ (thereby reducing the rank of $Z^{k+1}$) as long as the objective function value does not increase, while never moving partially within the current face.
\item[$\bullet$ \namedlabel{algIFOpt}{\textsc{IF-Optimization}}] \textsc{IF-Optimization} -- The simplified version of Algorithm \ref{RRa} with full in-face optimization as described at the end of Section \ref{qwerty}.  The in-face optimization subproblem is (approximately) solved using the proximal gradient method with matrix entropy prox function.
\item[$\bullet$ \namedlabel{algIFRank}{\textsc{IF-Rank-Strategy}}] \textsc{IF-Rank-Strategy} -- Algorithm \ref{RRa} with the away-step strategy and with $\gamma_1$ and $\gamma_2$ set dynamically as follows: we initially set $\gamma_1 = \gamma_2 = \infty$, and then reset $\gamma_1 = \gamma_2 = 1$ after we observe five consecutive iterations where $\rnk(Z^k)$ does not increase. This version can be interpreted as a two-phase method where we run Algorithm \ref{algo-vfw} until we observe that $\rnk(Z^k)$ begins to ``stall," at which point we switch to Algorithm \ref{RRa} with $\gamma_1 = \gamma_2 = 1$.
\item[$\bullet$ \namedlabel{algAwayNat}{\textsc{FW-Away-Natural}}] \textsc{FW-Away-Natural} -- Algorithm \ref{GMalgorithm}
\item[$\bullet$ \namedlabel{algAwayAtom}{\textsc{FW-Away-Atomic}}] \textsc{FW-Away-Atomic} -- Algorithm \ref{GMalgorithm} applied to the atomic reformulation of \eqref{easy} using the right-side of \eqref{big-simplex} \cite{simon, beck2015linearly, pena2015neumann}.
\end{itemize}
Tables \ref{small-table}, \ref{big-table}, and \ref{movie-table} present our aggregate computational results.  Before discussing these in detail, it is useful to first study Figure \ref{fig:rank-gap} which shows the behavior of each method in terms of ranks of iterates\footnote{The rank of a matrix is computed as the number of singular values larger than $10^{-6}$.  The rank-$1$ SVD computation for equation \eqref{easy2} is performed using the Matlab function \texttt{eigs}.} (left panel) and relative optimality gap (right panel) as a function of run time, for a particular (and very typical) simulated instance.  Examining the rank plots in the left panel, we see that the evolution of $\rnk(Z^k)$ is as follows:  the four methods \ref{algIFOneOne}, \ref{algIFZeroOne}, \ref{algIFZeroInf}, and \ref{algAwayNat} all quickly attain $\rnk(Z^k) \approx 37$ (the apparent rank of the optimum) and then stay at or near this rank from then on.  In contrast, the four methods \ref{algFW}, \ref{algIFRank}, \ref{algIFOpt} and \ref{algAwayAtom} all grow $\rnk(Z^k)$ approximately linearly during the early stages (due to a larger percentage of regular Frank-Wolfe steps), and then reach a maximum value that can be an order of magnitude larger than the optimal rank before the rank starts to decrease.  Once the rank starts to decrease, \ref{algIFRank} and \ref{algIFOpt} decrease $\rnk(Z^k)$ rather rapidly, whereas \ref{algFW} and \ref{algAwayAtom} decrease $\rnk(Z^k)$ painfully slowly.

The right panel of Figure \ref{fig:rank-gap} shows the relative optimality gaps of the methods.  It is noteworthy that two methods -- \ref{algIFOpt} and \ref{algIFRank} -- achieve very rapid progress during their early stages, a point that we will soon revisit.  However, all methods exhibit eventual slow convergence rates which is in line with the $O(1/k)$ theoretical convergence bound.

Let us now synthesize the two panels of Figure \ref{fig:rank-gap}.  The four methods \ref{algFW}, \ref{algIFRank}, \ref{algIFOpt} and \ref{algAwayAtom} each go through two phases: in the first phase each constructs a ``high information'' (high rank) solution (by taking mostly regular Frank-Wolfe steps), followed by a second phase where the solution is ``refined'' by lowering the rank while further optimizing the objective function (by taking proportionally more away-steps).  \ref{algFW} and \ref{algAwayAtom} build up to very high information but their build-down is sorely ineffective both in terms of ranks and objective function values.  \ref{algIFRank} is extremely effective at the refinement phase, and \ref{algIFOpt} is less effective in terms of rank reduction but still moreso than the other methods except of course for \ref{algIFRank}.  The other four methods, namely \ref{algIFOneOne}, \ref{algIFZeroOne}, \ref{algIFZeroInf}, and \ref{algAwayNat}, all rarely exceed rank $37$, as they spend a very high proportion of their effort on away-steps.  Of these four methods, \ref{algIFZeroInf} tends to perform best in terms of objective function values, as will be seen shortly in Tables \ref{small-table} and \ref{big-table}.  Last of all, we point out that for very large-scale problems storing the SVD of a high-rank matrix may become burdensome (over and above the computational cost for computing in-face directions on high-dimensional faces); thus it is important that the maximum rank of the iterates be kept small. In this regard Figure \ref{fig:rank-gap} indicates that excessive memory may arise for \ref{algFW}, \ref{algIFRank}, \ref{algAwayAtom}, and possibly \ref{algIFOpt}.
\begin{figure}[h!]
\centering
\scalebox{1}{\begin{tabular}{c c}
%\multicolumn{2}{c}{\sf Rank and Optimality Gap vs.\ Time for Various Methods \medskip}\\
\scriptsize{\sf {$\text{Rank}(Z^k)$ vs.\ Time (secs)}} & \scriptsize{\sf{$\text{Log}_{10}\left(\tfrac{f(Z^k) - f^\ast}{f^\ast}\right)$ vs.\ Time (secs)}} \\

%\rotatebox{90}{\sf {\scriptsize {~~~~~~~~~~~~~~~~~~~~~~~~~~~~~~~~Rank}}} &
\includegraphics[width=0.43\textwidth,height=0.43\textheight, keepaspectratio=true]{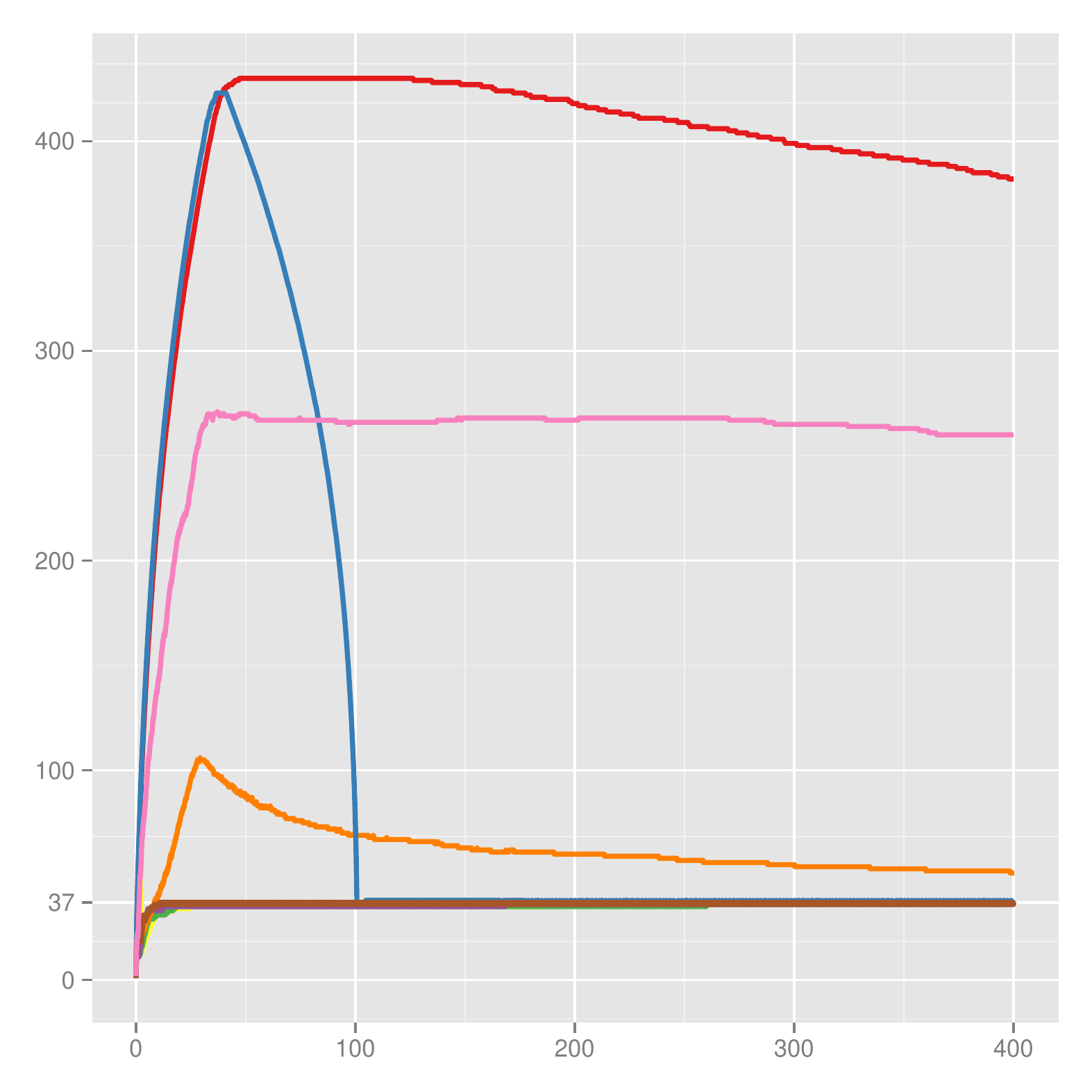}
&

%\rotatebox{90}{\sf {\scriptsize {~~~~~~~~~~~~~~~~~~$\text{Log}_{10}$(Optimality Gap)}}} &
\includegraphics[width=0.43\textwidth,height=0.43\textheight, keepaspectratio=true]{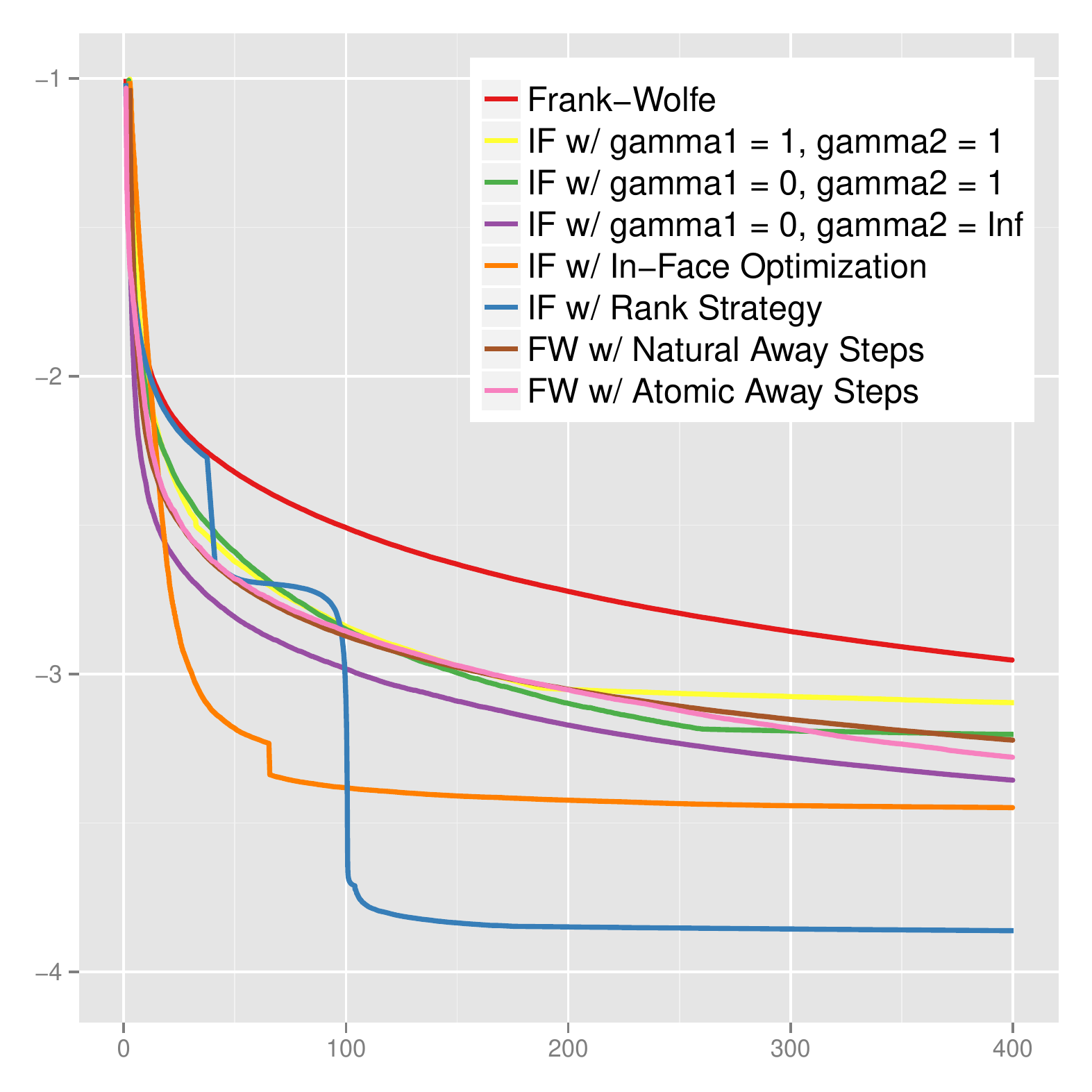} \\

%& \scriptsize{\sf {CPU Time (secs)}}  & \scriptsize{\sf {CPU Time (secs)}} \\
\end{tabular}}
\caption{Figure showing plots of rank and relative optimality gap (log-scale) versus run time for different methods/strategies, for a single randomly generated problem instance with $m = 2000$, $n = 2500$, $\rho = 0.01$, $r = 10$, $\text{SNR} = 4$, and $\delta = 8.01$. This problem has a (very nearly) optimal solution with rank 37.}\label{fig:rank-gap}
\end{figure}

Table \ref{small-table} presents computational results for three different types of small-scale examples, averaged over $25$ sample instances generated and run for each type.  Note that the run time, final rank, and maximum ranks reported in Table \ref{small-table} are in synch with the patterns observed in Figure \ref{fig:rank-gap}.  \ref{algIFRank} exhibits the best run times, followed by \ref{algIFOpt} and then by \ref{algIFZeroInf}, all of which significantly outperform \ref{algFW}, \ref{algAwayNat}, and \ref{algAwayAtom}.  Furthermore, \ref{algIFOpt} and \ref{algIFZeroInf} have relatively low values of the maximum rank (unlike \ref{algIFRank}), while not giving up too much in terms of run time relative to \ref{algIFRank}.  Finally, note that \ref{algAwayAtom} is dramatically ineffective at delivering low rank solutions, which is undoubtedly due to the fact that the faces of the atomic representation are simply too small to be effective -- see Figure \ref{fig:atomic} and the discussion at the end of Section \ref{fw}.

Table \ref{big-table} presents computational results for eight individual medium-large scale examples.  Here we see mostly similar performance of the different methods as was see for the small-scale examples in Table \ref{small-table}.  \ref{algIFRank}, \ref{algIFZeroInf}, and \ref{algIFOpt} deliver the best balance between final rank, maximum rank, and run time, with perhaps \ref{algIFZeroInf} delivering consistently lower rank solutions albeit with higher run times.  We note that for these instances \ref{algIFRank} does not consistently deliver low rank solutions, which is due to the extra time it takes before the second phase (``refinement'') of the method commences.

\begin{table}[]
\centering
\resizebox{\textwidth}{!}{%
\begin{tabular}{@{}ccccccccccccc@{}}
\toprule
\multicolumn{13}{c}{\bf{Small-Scale Examples (25 samples per example)} \medskip} \\
 &  &  & & & \multicolumn{5}{c}{\bf In-Face Extended FW (IF-$\ldots$)} &  & \multicolumn{2}{c}{\bf Away Steps} \\
\cline{6-10} \cline{12-13}
 &  &  & {\bf Regular} & & \multicolumn{3}{c}{\bf $\gamma_1, \gamma_2$} & \bf{In-Face}   & \bf{Rank}   &    & & \\
 \cline{6-8}
\bf{Data}                             &  & \bf{Methods}         & \bf{FW} &  & \bf{1,1}   & \bf{0,1}   & \bf{0,$\infty$} & \bf{Opt.}   & \bf{Strategy}   & & \bf{Natural}    & \bf{Atomic} \\ \midrule \\
 $ m = 200, n = 400, \rho =     0.10 $ & & Time (secs) &    29.51 & &    22.86 &    23.07 &     7.89 &     \bf{2.34} &     \bf{2.30} & &    14.71 &     6.21 \\
$ r = 10, \text{SNR} = 5, \delta_{\text{avg}} =     3.75 $ & & Final Rank & 118.68 & & 16.36 & 16.36 & \bf{16.44} & 29.32 & 28.20 & & 16.72 & 119.00 \\
 & & Maximum Rank & 146.48 & & 19.04 & 17.28 & \bf{17.56} & 32.08 & 145.20 & & 18.04 & 121.96 \\ \\
$ m = 200, n = 400, \rho =     0.20 $ & & Time (secs) &   115.75 & &   153.42 &   150.89 &    27.60 &    20.62 &     \bf{3.48} & &    50.52 &    24.52 \\
$ r = 15, \text{SNR} = 4, \delta_{\text{avg}} =     3.82 $ & & Final Rank & 96.44 & & 16.16 & 16.12 & \bf{16.52} & 19.88 & 21.24 & & 16.68 & 106.60 \\
 & & Maximum Rank & 156.52 & & 26.72 & 17.96 & \bf{17.80} & 31.48 & 160.36 & & 18.84 & 106.80 \\ \\
$ m = 200, n = 400, \rho =     0.30 $ & & Time (secs) &   171.23 & &   198.96 &   202.01 &    35.93 &    31.67 &     \bf{5.04} & &    66.22 &    67.72 \\
$ r = 20, \text{SNR} = 3, \delta_{\text{avg}} =     3.63 $ & & Final Rank & 91.80 & & 20.08 & 20.08 & \bf{20.60} & \bf{21.72} & 25.56 & & 20.44 & 94.64 \\
 & & Maximum Rank & 162.24 & & 25.80 & 22.04 & \bf{21.96} & 33.36 & 168.72 & & 22.16 & 95.08 \\ \\  \bottomrule
\end{tabular}
}
\caption{Table reports the time required for each method to reach a relative optimality gap of $10^{-2.5}$, the rank of the corresponding final solution, and the maximum rank of the iterates therein. Numbers highlighted in boldface indicate the methods that perform well with regard to each criteria, while not performing poorly on run time.   All results are averaged over 25 samples for each problem type.}
\label{small-table}
\end{table}

\begin{table}[]
\centering
\resizebox{\textwidth}{!}{%
\begin{tabular}{@{}ccccccccccccc@{}}
\toprule
\multicolumn{13}{c}{\bf{Medium-Large Scale Examples} \medskip} \\
 &  &  & & & \multicolumn{5}{c}{\bf In-Face Extended FW (IF-$\ldots$)} &  & \multicolumn{2}{c}{\bf Away Steps} \\
\cline{6-10} \cline{12-13}
 &  &  & {\bf Regular} & & \multicolumn{3}{c}{\bf $\gamma_1, \gamma_2$} & \bf{In-Face}   & \bf{Rank}   &    & & \\
 \cline{6-8}
\bf{Data}                             &  & \bf{Methods}         & \bf{FW} &  & \bf{1,1}   & \bf{0,1}   & \bf{0,$\infty$} & \bf{Opt.}   & \bf{Strategy}   & & \bf{Natural}    & \bf{Atomic} \\ \midrule \\
$ m = 500, n = 1000, \rho =     0.25$ & & Time (secs) &   137.62 & &    51.95 &    53.21 &    18.20 &     \bf{4.41} &     \bf{6.37} & &    31.55 &   157.31 \\
$ r = 15, \text{SNR} = 2, \delta =     3.57 $ & & Final Rank (Max Rank) & 53 (126) & & 16 (17) & 15 (17) & 16 (17) & \bf{17 (19)} & 121 (136) & & 15 (17) & 50 (52) \\ \\
$ m = 500, n = 1000, \rho =     0.25 $ & & Time (secs) &   256.08 & &   110.37 &   110.77 &    46.07 &     \bf{6.76} &     \bf{7.91} & &    73.95 &   322.24 \\
$ r = 15, \text{SNR} = 10, \delta =     4.11 $ & & Final Rank (Max Rank) & 41 (128) & & 15 (17) & 15 (17) & 16 (17) & \bf{15 (18)} & {\bf 18} (140) & & 16 (17) & 48 (48) \\ \\
$ m = 1500, n = 2000, \rho =     0.05 $ & & Time (secs) &   124.76 & &   108.97 &   113.58 &    24.75 &    \bf{11.09} &    \bf{12.71} & &    40.23 &    60.83 \\
$ r = 15, \text{SNR} = 2, \delta =     6.01 $ & & Final Rank (Max Rank) & 169 (210) & & 15 (18) & 16 (17) & \bf{16 (16)} & 31 (44) & 206 (206) & & 16 (16) & 128 (138) \\ \\
$ m = 1500, n = 2000, \rho =     0.05 $ & & Time (secs) &   800.01 & &   518.72 &   496.08 &   166.01 &    \bf{21.90} &    \bf{31.41} & &   309.58 &   407.22 \\
$ r = 15, \text{SNR} = 10, \delta =     8.94 $ & & Final Rank (Max Rank) & 119 (266) & & 15 (17) & 15 (17) & 15 (17) & \bf{15 (23)} & {\bf 15} (256) & & 15 (18) & 172 (185) \\ \\
$ m = 2000, n = 2500, \rho =     0.01 $ & & Time (secs) &   105.44 & &    45.39 &    36.47 &    \bf{23.15} &    \bf{20.07} &    47.83 & &    30.07 &    26.92 \\
$ r = 10, \text{SNR} = 4, \delta =     7.92 $ & & Final Rank (Max Rank) & 436 (436) & & 37 (38) & 35 (38) & \bf{37 (38)} & 67 (107) & 430 (430) & & 37 (39) & 245 (276) \\ \\
$ m = 2000, n = 2500, \rho =     0.05 $ & & Time (secs) &    99.84 & &    51.90 &    48.26 &    18.79 &     \bf{6.92} &     \bf{6.70} & &    30.37 &    89.09 \\
$ r = 10, \text{SNR} = 2, \delta =     5.82 $ & & Final Rank (Max Rank) & 68 (98) & & 10 (11) & 10 (11) & \bf{11 (11)} & 13 (15) & 94 (94) & & 10 (11) & 52 (52) \\ \\
$ m = 5000, n = 5000, \rho =     0.01 $ & & Time (secs) &   251.33 & &   168.66 &   172.21 &    64.56 &    \bf{26.25} &    \bf{17.70} & &    96.79 &    90.41 \\
$ r = 10, \text{SNR} = 4, \delta =    12.19 $ & & Final Rank (Max Rank) & 161 (162) & & 10 (24) & 11 (18) & \bf{11 (20)} & 22 (34) & 20 (112) & & 10 (16) & 181 (182) \\ \\
$ m = 5000, n = 7500, \rho =     0.01 $ & & Time (secs) &   272.19 & &   107.19 &   116.58 &    \bf{52.65} &    \bf{54.02} &   145.13 & &   107.60 &    94.96 \\
$ r = 10, \text{SNR} = 4, \delta =    12.19 $ & & Final Rank (Max Rank) & 483 (483) & & 33 (43) & 34 (36) & \bf{32 (37)} & 63 (123) & 476 (476) & & 36 (42) & 229 (298) \\ \\ \bottomrule
\end{tabular}
}
\caption{Table reports the time required for each method to reach a relative optimality gap of $10^{-2.5}$, the rank of the corresponding final solution, and the maximum rank of the iterates therein, for eight single problem instances. Numbers highlighted in boldface indicate the methods that perform well with regard to each criteria, while not performing poorly on run time.}
\label{big-table}
\end{table}

Table \ref{movie-table} shows computational tests on a large-scale real dataset, namely the MovieLens10M dataset, with $m = 69878$, $n = 10677$, $|\Omega| = 10^7$ (with sparsity approximately $1.3\%$), and $\delta = 2.59$.  We only tested \ref{algIFZeroInf} (and benchmarked against \ref{algFW}) since \ref{algIFZeroInf} appears to be very promising for large-scale instances due to its ability to maintain relatively low-rank iterates throughout, while also performing well in terms of run time.  The results in Table \ref{movie-table} further reinforce the findings from Table \ref{big-table} concerning the advantages of \ref{algIFZeroInf} both in terms of rank of the final iterate as well as run time to achieve the target optimality gap.

\noindent{\bf Summary Conclusions.}  In addition to its theoretical computational guarantees (Theorem \ref{RRguarantee}, Proposition \ref{prop-alt}), the In-Face Extended Frank-Wolfe Method (in different versions) shows significant computational advantages in terms of delivering low rank and low run time to compute a target optimality gap.  Especially for larger instances, \ref{algIFZeroInf} delivers very low rank solutions with reasonable run times.  \ref{algIFRank} delivers the best run times, beating existing methods by a factor of $10$ or more.  And in the large-scale regime, \ref{algIFOpt} generally delivers both low rank and low run times simultaneously, and is usually competitive with the best methods on one or both of rank and run time.

\begin{table}[]
\centering
\resizebox{.8\textwidth}{!}{%
\begin{tabular}{@{}ccccccccc@{}}
\toprule
&\multicolumn{7}{c}{\bf{MovieLens10M Dataset}\medskip} &\\
& & & \multicolumn{2}{c}{\bf{\ref{algFW}}} & & \multicolumn{2}{c}{\bf{\ref{algIFZeroInf}}}& \\ \cline{4-5}\cline{7-8}
&Relative Optimality Gap & & Time (mins) & Rank & & Time (mins) & Rank& \\ \midrule
&$10^{-1.5}$ & & 7.38 & 103 & & 7.01 & 44& \\
&$10^{-2}$ & & 28.69 & 315 & & 14.73 & 79& \\
&$10^{-2.25}$ & & 69.53 & 461 & & 22.80 & 107& \\
$ \ \ \ \ \ \ \ \ \ \ \ \ \ \ \ \ \ $ &$10^{-2.5}$ & & 178.54 & 454 & & 42.24 & 138& $ \ \ \ \ \ \ \ \ \ \ \ \ \ \ \ \ $ \\ \bottomrule
\end{tabular}
}
\caption{CPU time and rank of final solutions for \ref{algFW} and \ref{algIFZeroInf} for different relative optimality gaps for the MovieLens10M dataset.}
\label{movie-table}
\end{table}
%\clearpage
\small{
\bibliographystyle{abbrv}
\bibliography{GF-papers-nips}
}
\normalsize{
\appendix
\section{Appendix - Remainder of the Proof of Theorem \ref{QAtheorem}}\label{swapoff}}
Recall that it remains to prove the following inequality:
\begin{equation}\label{snrc2}\frac{1}{f(x_{k+1}) - B_{k+1}} \ge \frac{1}{f(x_k) - B_k} + \frac{1}{2\bar L D^2}  \ \ \ \ \mathrm{for~all~} k \ge 0 \ . \end{equation}
Let us fix some simplifying notation.  Let $r_k := f(x_k) - B_k \ge 0$ and $G_k := \nabla f(x_k)^T(x_k - \tilde x_k) \ge 0$, and note that $B_k \ge B^w_k = f(x_k) - G_k$, so that $G_k \ge r_k \ge 0$ for $k \ge 0$.  Also define $C_k :=\bar L \|\tilde x_k - x_k\|^2$, whereby $C_k \le \bar L D^2$ and $\bar\alpha_k = \min\left\{\frac{G_k}{C_k}, 1 \right\}$ for $k \ge 0$.  With this notation \eqref{snrc2} can be written as $1/r_{k+1} \ge 1/r_k + 1/(2\bar L D^2)$.  Substituting $x=x_k$ and $y=x_{k+1} = x_k + \bar\alpha_k(\tilde x_k - x_k)$ in \eqref{lipschitz2} and using $\bar L \ge L$ yields:
\begin{equation}\label{vivek}
f(x_{k+1}) \le  f(x_{k})+\bar\alpha_k\nabla f(x_k)^T(\tilde x_k-x_k) + \tfrac{\bar L}{2}\bar\alpha_k^2 \|\tilde x_k - x_k\|^2 = f(x_{k})-\bar\alpha_k G_k + \tfrac{1}{2}\bar\alpha_k^2C_k \ . \end{equation}
Note that if we instead used an exact line-search to determine $x_{k+1}$, then \eqref{vivek} also holds since in that case we have $f(x_{k+1}) \leq f(x_k + \bar\alpha_k(\tilde x_k - x_k))$. We now examine two cases depending on the relative magnitudes of $G_k$ and $C_k$.
\noindent {\bf Case 1:  $G_k \le C_k$}.  In this case $\bar\alpha_k = G_k / C_k$, and substituting this value in the right side of \eqref{vivek} yields $f(x_{k+1}) \le f(x_k) - \frac{(G_k)^2}{2C_k}$, which shows that $f(x_{k+1}) \le f(x_k)$ as well as $r_{k+1} \le r_k$, and also yields:
$$
r_{k+1} \ = \ f(x_{k+1}) - B_{k+1} \le f(x_{k+1}) - B_{k} \le f(x_k) - \frac{(G_k)^2}{2C_k} - B_{k} = r_k -  \frac{(G_k)^2}{2C_k} \le  r_k -  \frac{r_k r_{k+1}}{2C_k} \ , $$ where the last inequality uses $r_{k+1} \le r_k \le G_k$.  Dividing the above inequality by $r_{k+1}r_k$ and rearranging yields
$$\frac{1}{r_{k+1}} \ge \frac{1}{r_k} + \frac{1}{2C_k} \ge \frac{1}{r_k} + \frac{1}{2\bar L D^2} \ , $$ where the second inequality above uses $C_k \le \bar L D^2$.  This shows that \eqref{snrc2} holds in this case.

\noindent {\bf Case 2:  $G_k > C_k$}.  In this case $\bar\alpha_k = 1$.  Substituting $x=x_k$ and $y=x_{k+1} = x_k + \bar\alpha_k(\tilde x_k - x_k) = \tilde x_k$ in \eqref{vivek} yields $f(x_{k+1}) \le f(x_k) - G_k + \frac{1}{2}C_k < f(x_k) - C_k + \frac{1}{2}C_k = f(x_k) -  \frac{1}{2}C_k$, which shows that $f(x_{k+1}) < f(x_k)$  as well as $r_{k+1} < r_k$, and also yields:
\begin{equation}\label{patriots}
r_{k+1} \ = \ f(x_{k+1}) - B_{k+1} \le  f(x_{k+1}) - B_{k} \le  f(x_{k}) -G_k + \tfrac{1}{2}C_k - B_{k} = r_k -  G_k + \tfrac{1}{2}C_k \ , \end{equation} from which we derive:
\begin{equation}\label{eviesleep} 0 \le r_{k+1} \le r_k -  G_k + \tfrac{1}{2}C_k < r_k -  G_k + \tfrac{1}{2}G_k = r_k - \tfrac{1}{2}G_k \ , \end{equation} where the last inequality above uses $G_k > C_k$.  We now consider two sub-cases, one for $k=0$ and another sub-case for $k \ge 1$.  Let us first consider when $k=0$.  Then
$$ G_0 r_0 +G_0 C_0 = G_0 r_0 +\tfrac{1}{2} G_0 C_0 +\tfrac{1}{2} G_0 C_0 \ge (r_0)^2 +\tfrac{1}{2} (C_0)^2 +\tfrac{1}{2} r_0 C_0 \ ,  $$ since $G_0 \ge C_0$ and $G_0 \ge r_0$, and now add $r_0C_0$ to both sides and rearrange to yield:
$$r_0C_0 \ge r_0C_0 + (r_0)^2   - G_0r_0-G_0C_0 +\tfrac{1}{2} (C_0)^2 +\tfrac{1}{2} r_0C_0 =(r_0-G_0+\tfrac{1}{2}C_0)(r_0+C_0) \ge r_1(r_0+C_0)   \ , $$ where the second inequality uses \eqref{patriots} with $k=0$.  Therefore:
$$\frac{1}{r_1} \ge \frac{r_0+C_0}{r_0C_0} = \frac{1}{r_0} + \frac{1}{C_0} \ge \frac{1}{r_0} + \frac{1}{\bar L D^2} \ , $$ which proves \eqref{snrc2} for this case for $k = 0$.  Last of all, we consider when $k \ge 1$.  Taking  \eqref{eviesleep} and dividing by $r_k r_{k+1}$ and rearranging yields:
$$\frac{1}{r_{k+1}} > \frac{1}{r_k} + \frac{G_k}{2r_k r_{k+1}} \ge \frac{1}{r_k} + \frac{1}{2 r_{k+1}}  \ , $$ where the second inequality follows since $G_k \ge r_k$.  Now notice from \eqref{eviesleep} that $r_{k+1} \le r_k -G_k + C_k/2 \le C_k/2$ since $G_k \ge r_k$.  Substituting this last inequality into the right-most term above yields:
$$\frac{1}{r_{k+1}} \ge \frac{1}{r_k} + \frac{1}{C_k} \ge \frac{1}{r_k} + \frac{1}{2C_k} \ge \frac{1}{r_k} + \frac{1}{2\bar L D^2} \ , $$ which shows \eqref{snrc2} for this case for $k \ge 1$, and completes the proof. \qed

\end{document}